\documentclass[11pt, oneside]{article}   	% use "amsart" instead of "article" for AMSLaTeX format
\usepackage{geometry}                		% See geometry.pdf to learn the layout options. There are lots.
\usepackage{graphicx}				% Use pdf, png, jpg, or eps? with pdflatex; use eps in DVI mode
\usepackage{amssymb}
\usepackage{amsmath}
\usepackage{amsthm}
\usepackage{hyperref}
\usepackage{mathtools}
\mathtoolsset{showonlyrefs=true}
\usepackage{mathrsfs}
\usepackage{graphicx,xcolor}
\usepackage{pgfplots,tikz}
\usepackage{enumerate, comment}
\title{Sharp conditions for exponential and non-exponential uniform stabilization of the time-dependent damped wave equation}
\author{ Perry Kleinhenz}%
\date{}							% Activate to display a given date or no date

\theoremstyle{definition}
\newtheorem{definition}{Definition}
\newtheorem{example}[definition]{Example}

\newtheorem{assumption}{Assumption}
\newtheorem{remark}{Remark}
\theoremstyle{theorem}
\newtheorem{theorem}{Theorem}[section]
\newtheorem{lemma}[theorem]{Lemma}

\newtheorem{proposition}[theorem]{Proposition}

\newcommand{\Rb}{\mathbb{R}}

\newcommand{\Rn}{\mathbb{R}^n}

\newcommand{\C}{\mathbb{C}}

\newcommand{\Nb}{\mathbb{N}}

\newcommand{\intrn}{\int_{\Rn}}

\newcommand{\ra}{\rightarrow}

\newcommand{\rhu}{\rightharpoonup}
\newcommand{\<}{\left\langle}
\renewcommand{\>}{\right\rangle}
\newcommand{\e}{\varepsilon}
\renewcommand{\d}{\delta}

\newcommand{\limn}{\lim_{n \ra \infty}}

\newcommand{\limj}{\lim_{j \ra \infty}}
\newcommand{\limt}{\lim_{t \ra \infty}}
\newcommand{\nm}[1]{\left\| #1 \right\|}

\newcommand{\lp}[2]{ \nm{#1}_{L^{#2}}}

\newcommand{\hp}[2]{\nm{#1}_{H^{#2}}}

\newcommand{\ltwo}[1]{\lp{#1}{2}}

\newcommand{\ltwooo}[1]{\nm{ #1}_{L^2_0}}

\newcommand{\ltwom}[1]{\| #1\|_{L^2(M)}}

\newcommand{\vphi}{\varphi}

\newcommand{\p}{\partial}

\newcommand{\Ci}{C^{\infty}}

\newcommand{\Ac}{\mathcal{A}}
\newcommand{\Lc}{\mathcal{L}}

\newcommand{\Hc}{\mathcal{H}}

\newcommand{\ti}{\widetilde}

\renewcommand{\Re}{\text{Re }}

\newcommand{\Cm}{\overline{C}}

\newcommand{\Winf}{W_{\infty}}
\newcommand{\Omegab}{\overline{\Omega}}
\newcommand{\Char}{\text{Char}}
\newcommand{\Gc}{\mathcal{G}}
\newcommand{\inter}{\text{int}}
\newcommand{\Mc}{\mathcal{M}}
\newcommand{\Bc}{\mathscr{B}}

\def\XXint#1#2#3{{\setbox0=\hbox{$#1{#2#3}{\int}$ }
\vcenter{\hbox{$#2#3$ }}\kern-.6\wd0}}

\begin{document}
\maketitle
\begin{abstract}
It is classical that uniform stabilization of solutions to the autonomous damped wave equation is equivalent to every geodesic meeting the positive set of the damping, which is called the geometric control condition. In this paper, it is shown that for time-dependent damping a generalization of the geometric control condition is equivalent to uniform stabilization at an exponential rate. Additionally, upper and lower bounds on non-exponential uniform stabilization rates are computed for damping which do not satisfy this geometric control condition. Decay rates are guaranteed via an observability argument, including a new uniform time-dependent observability inequality. Bounds on decay rates are proved via a Gaussian beam construction.
\end{abstract}
\section{Introduction}
Let $(M,g)$ be a smooth compact manifold without boundary and let $\Delta_g$ be the associated Laplace-Beltrami operator. Let $W \in L^{\infty} (M\times [0,\infty))$ be a nonnegative function. Consider the damped wave equation with time-dependent damping
\begin{equation}\label{TDWE}
\begin{cases}
(\p_t^2 -\Delta_g + 2W(x,t)\p_t)u=0, \\
(u,u_t)|_{t=0} =(u_0, u_1) \in H^1(M) \times L^2(M).
\end{cases}
\end{equation}
The main of object of study in this paper is the energy of a solution
$$
E(u,t) = \frac{1}{2} \int_M |\nabla_g u(x,t)|^2 + |\p_t u(x,t)|^2 dx_g.
$$
Uniform stabilization occurs when there exists a function $r(t)\ra 0$ as $t \ra \infty$, such that  \begin{equation}
	E(u,t) \leq r(t) E(u,0).
\end{equation} 
When $W$ is autonomous, exponential uniform stabilization is equivalent to $W$ satisfying the Geometric Control Condition (GCC) \cite{Ralston1969, RauchTaylor1975}. The GCC is satisfied if there exists some $L>0,$ such that every geodesic with length at least $L$ intersects the set $\{W>0\}$. 
As introduced in \cite{Lebeau1996}, the GCC is equivalent to the existence of $T_0,\Cm>0,$ such that for all unit-speed geodesics $\gamma(t)$ and $T \geq T_0$
\begin{equation}
	\frac{1}{T}\int_0^T W(\gamma(t)) dt \geq \Cm.
\end{equation} 
That is, there is a uniform lower bound on the average of the damping along any geodesic of sufficient length. In this paper, we show that the appropriate generalization of this condition to the time-dependent setting is equivalent to exponential uniform stabilization. %for sufficiently smooth damping. 

\begin{assumption}\label{TGCC} (Time-dependent geometric control condition)
Assume there exists $T_0>0$, $ \Cm >0,$ such that for all unit-speed geodesics $\gamma(t)$, any starting time $t_0 \in [0,\infty),$ and $T \geq T_0$
$$
\frac{1}{T} \int_0^T W(\gamma(t), t_0 + t) dt \geq \Cm.
$$
\end{assumption}
\begin{theorem}\label{mainresult}
Suppose $W(x,t) \in C^0_u(M \times [0,\infty))$, that is $W$ is uniformly continuous and uniformly bounded. There exists $C,c>0$ such that all solutions to \eqref{TDWE} satisfy
$$
E(u,t+t_0) \leq C e^{-ct} E(u,t_0), \quad \text{for all } t_0, t \geq 0,
$$
if and only if $W$ satisfies Assumption \ref{TGCC}.
\end{theorem}
Theorem \ref{mainresult} is actually a consequence of two more general results: Theorem \ref{lowerboundthm} which gives lower bounds on uniform stabilization rates when $W \in C^0(M \times [0,\infty))$ and Theorem \ref{bdecaythm} which provides uniform stabilization rates when $W$ possibly goes to zero as $t \ra \infty$ and $M$ may have a boundary. 
%For $W$ autonomous, solutions are a semigroup. Thus, when uniform stabilization occurs, the rate can always be taken $Ce^{-ct}$ for some $C,c>0$ and so the analogous result totally classifies uniform stabilization for $W$ autonomous. However when $W$ depends on time, solutions no longer form a semigroup and so in the absence of the TGCC uniform stabilization could occur with a non-exponential rate. The other main result of this paper is that non-exponential uniform stabilization indeed occurs.

To discuss lower bounds on uniform stabilization rates, we first define the smallest average of $W$ along a unit speed geodesic $\gamma$ on time intervals of length $T$ as $L(T)$, and the limit of this quantity as $L_{\infty}$ 
\begin{equation}
	L(T) = \inf_{t_0 \geq 0, \gamma} \frac{1}{T}  \int_{t_0}^{t_0+T} W(\gamma(t), t) dt, \quad L_{\infty} = \lim_{T \ra \infty} L(T).
\end{equation}
We also define the smallest total amount of damping along any unit-speed geodesic from time $0$ to $t$ as $\Sigma(t)$
\begin{equation}
	\Sigma(t) = \inf_{\gamma} \int_0^t W(\gamma(s), s) ds.
\end{equation}
%Note that the existance of a $\Cm>0$ such that $\sigma(t) = \Cm t$ for $t$ large enough is equivalent to Assumption \ref{TGCC}. 
The quantity $L_{\infty}$ bounds the exponential uniform stabilization rate when the damping satisfies Assumption \ref{TGCC}, while $\Sigma(t)$ bounds the uniform stabilization rate in general. 

\begin{theorem}\label{lowerboundthm}
Assume $W \in C^0(M \times [0,\infty))$. 
\begin{enumerate}
	\item Define the best exponential decay rate 
	\begin{equation}
	\alpha = \sup\{\beta \geq 0: \exists C>0 \text{ such that } E(u, t+t_0) \leq C e^{-\beta t} E(u,t_0), \forall u \text { solving } \eqref{TDWE}, t_0,t \geq 0\}.
	\end{equation}
	Then $\alpha \leq 2 L_{\infty}$. 
	\item When $\Sigma(t) \ra \infty$, define 
	\begin{equation}
		\sigma = \sup \{\beta \geq 0: \exists C>0 \text{ such that } E(u,t) \leq C e^{-\beta \Sigma(t)} E(u,0) , \forall u \text{ solving } \eqref{TDWE}, t \geq 0\}.
	\end{equation}
	Then $\sigma \leq 2$.  
	\item If $\Sigma(t)$ is bounded, then there can not be a uniform stabilization rate. That is, there does not exist $r(t) \ra 0$ as $t \ra \infty,$ such that $E(u,t) \leq r(t) E(u,0)$.
\end{enumerate}
\end{theorem}
\begin{remark} 
	\begin{enumerate}
		\item The inequalities $\alpha \leq 2 L_{\infty}$ and $\sigma \leq 2$ are lower bounds on the uniform stabilization rate.
		\item If Assumption \ref{TGCC} is not satisfied, then $L_{\infty}=0,$ and exponential decay cannot occur. Although the form and content of parts 1) and 2) of this theorem are similar, they are written separately to emphasize the connection to best exponential decay rate results \cite{Lebeau1996, KeelerKleinhenz2023}, and because doing so makes it clear that Assumption 1 is necessary for exponential uniform stabilization.
		\item 	Note that $W \in C^0(M \times [0,\infty))$ need not be bounded as $t \ra \infty$. 
	\end{enumerate}
\end{remark}

Now to discuss uniform stabilization rates on manifolds with boundary, we begin with some definitions. Let $\Omega$ be an open bounded connected subset of $M$, with a smooth boundary if $\p \Omega \neq \emptyset$. When $\p \Omega$ is nonempty let $B$ either be the Dirichlet trace operator, $B u = u |_{\p \Omega},$ or the Neumann trace operator, $Bu = \p_n u|_{\p \Omega}$,
where $\p_n$ is the outward normal derivative on $\p \Omega$.

Let $A=-\Delta_g$ be the Laplace operator with domain
$$
D(A) = \{u \in L^2(\Omega); Au \in L^2(\Omega) \text{ and } Bu=0 \text{ when } \p\Omega \neq \emptyset\}.
$$
Note $A$ is self adjoint and nonnegative. %With Dirichlet boundary conditions $D(A)=H^2(\Omega) \cap H^1_0(\Omega)$. With Neumann boundary conditions, or when there is no boundary, $D(A)=\{u \in H^2: \frac{\p u}{\p n}|_{\p\Omega} =0\}$. 
With Dirichlet boundary conditions define $H = H^1_0(\Omega)$, and it's dual with respect to $L^2$ as $H'=H^{-1}(\Omega)$. With Neumann boundary conditions, or when $\p\Omega=\emptyset$, let $H=H^1(\Omega)$, and let $H'$ be its dual with respect to $L^2$.

Let $W \in C^{0}_u ( \overline{\Omega} \times [0,\infty))$, that is $W$ is uniformly continuous and uniformly bounded, up to the boundary if there is one. Consider the damped wave equation 
\begin{equation}\label{TDWEb}
\begin{cases}
(\p_t^2 + A+ 2 W(x,t) \p_t ) u =0,&(x,t) \in \Omega \times (0,\infty)\\
Bu = 0, &\text{when } \p\Omega \neq \emptyset  \\
(u,\p_t u)|_{t=0} = (u_0, u_1) &\in H \times L^2(\Omega).
\end{cases}
\end{equation}
Note that for any $(u_0, u_1) \in H \times L^2(\Omega)$ there is a unique weak solution $u \in L^2((0,\infty); H)$ of \eqref{TDWEb} with $\p_t u \in L^2((0,\infty); L^2(\Omega))$.

When there is a boundary the appropriate generalization of the GCC uses generalized geodesics, and is otherwise identical to Assumption \ref{TGCC}. See Appendix \ref{nullbichar} for details on the construction of generalized geodesics. 
\begin{assumption}\label{TGCCp} (Time-dependent geometric control condition on manifold with boundary)
Assume there exists $T_0, \Cm >0,$ such that for any unit-speed generalized geodesic $\gamma(t)$, any starting time $t_0 \in [0,\infty),$ and $T \geq T_0$
\begin{equation}
	\frac{1}{T} \int_0^T W(\gamma(t), t_0 + t) dt \geq \Cm.
\end{equation}
\end{assumption}
Now we can state our main uniform stabilization result. When $W(x,t)$ is equal to the product of a function $\ti{W}(x,t)$, satisfying this time-dependent geometric control condition, and a positive function $f(x,t)$, which potentially goes to $0$ as $t \ra \infty$, we obtain a uniform stabilization rate. Furthermore, the rate qualitatively matches that of Theorem \ref{lowerboundthm} case 2.

\begin{theorem}\label{bdecaythm}
Let $f \in L^{\infty}(\overline{\Omega} \times [0,\infty))$ and assume there exists $C_m, C_M>0$, and $b(t)$ a positive decreasing function, such that $C_m b(t) \leq f(x,t) \leq C_M b(t)$. Let $\ti{W} \in C^0_u( \overline{\Omega} \times [0,\infty))$ and  $W(x,t)=\ti{W}(x,t) f(x,t)$. 
When $\p\Omega \neq \emptyset$, assume that no generalized bicharacteristic has contact of infinite order with $ \p\Omega \times (0,\infty)$, that is $\Gc^{\infty}=\emptyset$. If $\ti{W}$ satisfies Assumption \ref{TGCCp}, then there exists $C,c>0,$ such that all solutions to \eqref{TDWEb} satisfy
\begin{equation}
	E(u,t) \leq C \exp(-c \Sigma(t)) E(u,0), \quad t \geq 0.
\end{equation}
\end{theorem}
See Appendix \ref{nullbichar} for a precise definition of $\Gc^{\infty}$. 

\begin{remark} 
	\begin{enumerate}
		\item When $f(x,t) =1$, $\ti{W}=W$, then Assumption \ref{TGCCp} guarantees $\Sigma(t) \geq \Cm t$ and Theorem \ref{bdecaythm} provides uniform stabilization at an exponential rate. Together with Theorem \ref{lowerboundthm} part 1, this proves Theorem \ref{mainresult} when $W \in C^0_u(M \times [0,\infty)) $ and $M$ does not have a boundary. Note that Assumption \ref{TGCCp} holds for any starting time $t_0 \geq 0$, which is what allows us to take any $t_0 \geq 0$ in Theorem \ref{mainresult}.
		
		\item When $M$ does not have a boundary, Theorem \ref{lowerboundthm} prevents decay faster than $\exp(-2\Sigma(t))$, while for $W$ decreasing of the form specified by Theorem \ref{bdecaythm}, decay is guaranteed at $\exp(-c\Sigma(t))$. This shows these rates are qualitatively correct. 

		\item Note that $f(x,t)$ need not be monotone decreasing, it only needs to be bounded between constant multiples of a monotone decreasing positive function $b(t)$.

		\item When $W$ is autonomous, solutions are a semigroup and uniform stabilization is automatically exponential, so the geometric control condition is equivalent to uniform stabilization. However, when $W$ depends on time, solutions no longer form a semigroup. Furthermore, non-exponential uniform stabilization rates occur and the time dependent geometric control condition is not equivalent to uniform stabilization. For example when $b(t)=(t+1)^{-\beta}$ for $\beta \in (0,1)$, Theorems \ref{lowerboundthm} and \ref{bdecaythm} give uniform stabilization at rate $\exp(-ct^{1-\beta})$. See Section \ref{examplesection} for a further discussion and examples. 
	\end{enumerate}
\end{remark}
Theorem \ref{bdecaythm} is proved via a new uniform time-dependent observability inequality for the wave equation. To discuss this in more detail, we first define the standard wave equation
\begin{equation}\label{WEb}
\begin{cases}
(\p_t^2 + A) \psi =0,&(x,t) \in \Omega \times (0,\infty)\\
B\psi = 0 &\text{when } \p\Omega \neq \emptyset  \\
(\psi,\p_t \psi)|_{t=0} = (\psi_0, \psi_1) &\in H \times L^2(\Omega).
\end{cases}
\end{equation}
When $W$ satisfies the time-dependent geometric control condition, Assumption \ref{TGCCp}, we prove an observability estimate with observability constant $C_1$ that does not depend on the starting time $t_0$.
\begin{proposition}\label{uniobserveprop}
Let $W \in C^0_u(\overline{\Omega}, [0,\infty))$. If $\p\Omega \neq \emptyset,$ assume that no generalized bicharacteristic has contact of infinite order with $\p\Omega \times (0,\infty)$, that is $\Gc^{\infty}=\emptyset$. If $W$ satisfies Assumption \ref{TGCCp}, then for each $T \geq T_0$, there exists $C_1>0,$ such that for all  $\psi$ solving the wave equation \eqref{WEb} and all starting times $t_0 \geq 0$, then
\begin{equation}\label{eq:uniformobserve}
%\frac{C_1 }{2} \left( \ltwo{\nabla \psi|_{t=t_0}}^2 + \ltwo{\p_t \psi_{t=t_0}}^2\right) = 
E(\psi,t_0) \leq C_1  \int_{t_0}^{t_0+T} \int_{\Omega} W |\p_t \psi|^2 dx dt. %C\int_0^T \ltwo{G \p_t u}^2 dt.
\end{equation}
\end{proposition}
\begin{remark}\label{r:data}
This result offers two improvements over the time-dependent observability inequality \cite[Theorem 1.8]{LRLTT}. First, the constant $C_1$ is uniform in the starting time $t_0$. As will be seen in Section \ref{s:bdecaythm}, this is the crucial ingredient in the proof of uniform stabilization rates for solutions to the damped wave equation \eqref{TDWEb}. Second, when $\p \Omega =\emptyset$ or the boundary condition is Neumann, Proposition \ref{uniobserveprop} removes technical assumptions on the initial data. In particular \cite{LRLTT} requires 
\begin{equation}
\int \psi_0 dx = \int \psi_1 dx =0.
\end{equation}
In that paper see page 5 and their definition of $L^2_0(\Omega)$. We do not impose such a requirement.
 
Note, it is possible to write a solution $\psi$ of \eqref{WEb}, as $\psi=\vphi+a+bt$, such that $\vphi$ also solves \eqref{WEb} and $\vphi, \p_t \vphi$ have zero mean. However, when plugging $\psi=\vphi+a+bt$ into \eqref{eq:uniformobserve}, notice 
\begin{equation}
	W|\p_t \psi|^2 = W|\p_t \vphi|^2+2 bW \p_t \vphi + W b^2.
\end{equation}
The term $W \p_t \vphi$ is not necessarily orthogonal to constants, so $2bW \p_t \vphi$ has an indefinite sign. So although the requirement of $\psi_0$ having mean zero can be removed in this manner, the requirement on $\psi_1$ cannot be removed in this same way.
%Removing this technical assumption on the initial data and providing a uniform constant makes this a true time-dependent generalization of the autonomous observability inequalities \cite{BardosLebeauRauch1992}. 
\end{remark}
%
%When the damping is the product of a damping satisfying Assumption \ref{TGCC} and a bounded, positive, non-increasing, function, the uniform stabilization rate is given by the exponential of this $\Sigma(t)$. 
%\begin{theorem}\label{decreasingtheorem}
%Let $f \in L^{\infty}(M \times [0,\infty))$ and assume there exists $C_m, C_M>0$, and $b(t)$ a positive non-increasing function such that $C_m b(t) \leq f(x,t) \leq C_M b(t)$. Let $\ti{W} \in C^0_u(M\times [0,\infty))$ and $W(x,t)=\ti{W}(x,t) f(x,t)$. If $\ti{W}$ satisfies Assumption \ref{TGCC}, then there exists $C,c>0$ such that all solutions to \eqref{TDWE} satisfy
%\begin{equation}
%E(u,t) \leq C E(u,0) \exp(-c\Sigma(t)).
%\end{equation}
%\end{theorem}
%When $\Sigma(t)=o(t)$, for example when $b(t)=(t+1)^{-\beta}$ for $\beta>0$, this gives uniform stabilization at a non-exponential rate. This non-exponential rate is qualitatively correct. It cannot be faster than $\exp(-2\Sigma(t))$. 

%Although Section 2 provides useful lower bound information for a broad class of damping, there is not an analogous general approach to guarantee a given decay rate. This is because the key intermediate estimate is an observability estimate for the wave equation, which does not have an explicit constant depending on $L(T)$. When there is no time dependence there is a formula for the constant which involves $L(T)$, but also the spectral gap of the stationary operator. 

\subsection{Functional Analysis Framework}
The equation \eqref{TDWE} can be rewritten as a system of first order equations 
$$
\p_t \begin{pmatrix} u \\ v \end{pmatrix} = \Ac(t) \begin{pmatrix} u \\ v \end{pmatrix}, \quad \Ac(t) = \begin{pmatrix} 0 & 1 \\ \Delta & -W(t) \end{pmatrix}, \quad U =\begin{pmatrix} u_0 \\ u_1 \end{pmatrix} \in \Hc,
$$
where $\Hc=H \times L^2(M)$. When $\Ac$ is autonomous it is the generator of a $C^0$ contraction semigroup $e^{t\Ac}$. Because $\nm{U}_{\Hc} \simeq E(u,t)^{\frac{1}{2}},$ uniform stabilization estimates can be obtained by estimating $\|e^{t\Ac}\|_{\Lc(\Hc)}$ via resolvent estimates on $\Ac$. When $\Ac(t)$ is time-dependent, there is no longer a semigroup structure, instead solutions are part of a uniform family, also called an evolution family. 

\subsection{Literature Review}
When $M$ is a manifold without boundary and $W$ is autonomous, \cite{RauchTaylor1975} show that the geometric control condition (GCC) implies exponential decay. The techniques of \cite{Ralston1969} can be used to show the GCC is necessary for exponential decay. The autonomous GCC results were extended to manifolds with boundary in  \cite{BardosLebeauRauch1992, BurqGerard1997}. Study of time-dependent damping using microlocal methods goes back to \cite{RauchTaylor1975b} where they show, if growing eigenmodes can be ruled out, then exponential decay of energy holds for time periodic damping satisfying a GCC hypothesis. 
Using a different approach \cite{LRLTT} show that time periodic $W$ satisfying a slightly different GCC hypothesis (Assumption \ref{ftGCC} in this paper) implies exponential decay and discuss several explicit examples. Note that any time periodic damping which satisfies the GCC hypotheses of \cite{LRLTT} also satisfies Assumption \ref{TGCCp}. For further discussion of time periodic damping see \cite{wirthperiodic, PaunonenSeifert2019}. This paper, and indeed \cite{BardosLebeauRauch1992} and \cite{LRLTT}, prove exponential stability for the damped wave equation via observability of the standard wave equation from $W$. This approach goes back to \cite{Haraux1989}. Applying this observability approach previously required time periodicity of the damping. This paper shows how to do so for general time dependence. 
For related work on time-dependent observability see \cite{Shao2019}.

For $W$ autonomous, the best possible exponential decay rate was computed by \cite{Lebeau1996}, in terms of long time averages of the damping and the spectral abscissa of the stationary operator. This was extended to damping given by a $0$th order pseudodifferential operator in \cite{KeelerKleinhenz2023}. The challenge in generalizing these results to time-dependent damping is finding a replacement for the stationary operator. Despite this, the Gaussian beam methods of those papers are used to prove the necessity of the time-dependent geometric control condition (TGCC) for exponential uniform stabilization in Theorem \ref{lowerboundthm}

%If this is because a trajectory $(\gamma(t), t)$ never intersects the positive set of the damping then it is not possible to obtain a uniform stabilization rate, see Remark 1. However, in some cases, as shown in \cite{PaunonenSeifert2019}, the solution decays to a periodic solution. For other time periodic damping see \cite{Wirthperiodic}.
One way for the TGCC to fail is for the damping to approach 0 as $t \ra \infty$. This is a natural setting to consider from a physical perspective. Damping materials cause decay of the mechanical energy of waves by converting kinetic energy into heat, and this heat can decrease the efficacy of the damping \cite{lz23, drda12}. 

Damping converging to $0$ as $t \ra \infty$ has been well-studied for $M=\Rn$. Results go back to \cite{Matsumura1977, Uesaka1980} which considered damping $W(x,t) \geq (t+1)^{-1}$ and obtained polynomial uniform stabilization rates. See also \cite{Mochizuki1976}. This was partially generalized in \cite{MochizukiNakazawa1996, MochizukiNakazawa2001}, which showed that damping satisfying
\begin{equation}
	W(x,t) \geq \frac{1}{(1+t) \log(1+t) \log(\log(1+t)) \cdots \log^{[n]}(1+t)},
\end{equation} 
must have decay, for some $C>0$
\begin{equation}
	\limt E(u,t) (\log^{[n]}(1+t))^{C}=0.
\end{equation}
Note that this is not a uniform stabilization rate. 

Improvements can be made when the damping is restricted to only depend on time. When $W(t)=\mu(t+1)^{-1},$ \cite{Wirth2004} obtained the sharp uniform stabilization rate, setting $\alpha = \min(\mu,2)$ 
\begin{equation}
	E(u,t) \leq E(u,0) (1+t)^{-\alpha}.
\end{equation}
This rate agrees with the qualitative polynomial decay guaranteed by Theorem \ref{bdecaythm}.
When $\limsup_{t \ra \infty} t W(t) < M,$ and the damping satisfies some symbol style estimates, then \cite{Wirth2006} obtained a uniform stabilization rate 
\begin{equation}
	E(u,t) \leq E(u,0) \exp(-\Sigma(t)),
\end{equation}
which agrees with Theorem \ref{bdecaythm}. 
On the other hand when $\limt t W(t) = \infty,$ and the damping satisfies the same symbol style estimates, \cite{Wirth2007} obtained a rate 
\begin{equation}
	\limt E(u,t)\left(\int_0^t \frac{ds}{W(s)}\right) =0.
\end{equation}
Note this is not a uniform stabilization rate. Furthermore, the rate provided by this result is weaker than Theorem \ref{bdecaythm}. For example with $W=(1+t)^{-\beta}$ with $\beta \in (0,1)$, this result provides polynomial decay, while Theorem \ref{bdecaythm} gives sub-exponential decay.
Analogous polynomial rates for $W(x,t) = a(x) b(t)$, with $b$ decreasing and $W(t) \simeq (1+t)^{-\beta}$ with $\beta \in (0,1)$ were proved in \cite{Kenigson2011}. 
The symbol assumptions of \cite{Wirth2006,Wirth2007} were relaxed in \cite{HirosawaWirth2008} by considering small oscillations around such functions. 
This was relaxed further to damping $W$, not necessarily decreasing, satisfying $(1+t)^{-\beta} \simeq W(t)$ for $\beta<1$ in \cite{vjdl}. Note that Theorem \ref{bdecaythm} provides decay rates for more general damping functions on manifolds. For example the dependence on $x$ and $t$ can be linked, $W$ can equal $0$ at some points, and the time-dependence need not be polynomial. Theorem \ref{lowerboundthm} bounds these decay rates from below. 

%An earlier version of this paper was in a separate preprint \cite{Kleinhenz2022}
\subsection{Outline}
In Section \ref{gaussianbeamsection}, Gaussian beams are used to obtain lower bounds on uniform stabilization rates, which are used to prove Theorem \ref{lowerboundthm}. 
In Section \ref{observesection} the uniform time-dependent wave equation observability inequality Proposition \ref{uniobserveprop} is proved.
In Section \ref{s:bdecaythm}, the uniform stabilization rate for damping potentially going to $0$ as $t \ra \infty$, Theorem \ref{bdecaythm}, is proved. 
In Section \ref{examplesection}, upper and lower bounds for uniform stabilization rates are proved for three  example classes of damping which fail to satisfy Assumption \ref{TGCCp}. Uniform stabilization rates are guaranteed using an observability inequality approach and lower bounds on these rates are obtained using Theorem \ref{lowerboundthm}.  In two of the three cases these bounds coincide in a manner that indicates the correct qualitative rate has been obtained. Appendix A provides necessary background on generalized null bicharacteristics and defect measures on manifolds with boundary. 
Appendix B contains a standard fact about weak continuity of solutions of the wave equation and a short-time observability inequality for the wave equation when the observation window is the whole manifold. %It could also be described as a short time energy equipartition result as the initial energy is estimated by the kinetic energy up to time $\d$. 

\textbf{Acknowledgments} I would like to thank Andr\'as Vasy for proposing this question to me and for his helpful comments throughout the course of this project. I would also like to thank the anonymous referees of the various versions of this paper for their thoughtful comments. I am also thankful to Jared Wunsch, Ruoyu P.T. Wang, Nicolas Burq, Emmanuel Tr\'elat, and Willie Wong for helpful conversations and correspondence.

\section{Lower bounds on uniform stabilization}\label{gaussianbeamsection}
In this section we recall the Gaussian beam construction of \cite{Ralston1982}. 
Gaussian beams are quasi-solutions $u_k$ of the wave equation which concentrate along a single geodesic $\gamma$. 
We next modify these Gaussian beams to produce quasi-solutions $v_{k,j}$ of the damped wave equation \eqref{TDWE}. 
To do so, we first approximate the damping $W$ by a smooth $W_j$, and then define $v_{k,j}$ by multiplying $u_k$ by a factor, $G_j(\gamma, t_0, t)$, which exponentially decreases according to the total amount of $W_j$ encountered along the geodesic $\gamma$ between time $t_0$ and $t$. 
As a result of this construction, these quasi-solutions of \eqref{TDWE} have energy converging to $G_j(\gamma,t_0,t)^2$ as $k \ra \infty$. 
We next show that the quasi-solutions $v_{k,j}$ of \eqref{TDWE} are ``nearby" to genuine solutions $\omega_k$ of \eqref{TDWE} with the same initial data. 
From this, we conclude that these genuine solutions have energy $E(\omega_k,t)$ close to $G(\gamma,t_0,t)^2$, which is the total amount of $W$ encountered along the geodesic $\gamma$. 
We finally connect $G(\gamma,t_0,t)$ to $L_{\infty}$ and $\Sigma(t)$ and use the constructed solutions of \eqref{TDWE} to prove Theorem \ref{lowerboundthm}.
\subsection{Gaussian beam and Quasi-solution construction}
We begin by recalling the Gaussian beam construction. Suppose $\Mc(t)$ is a symmetric $n \times n$ matrix-valued function with positive definite imaginary part. Let $\gamma(t)$ be a unit-speed geodesic in $\Rn$ and set 
$$
\psi(x,t) = \<\gamma'(t), x-\gamma(t)\> + \frac{1}{2} \<\Mc(t)(x-\gamma(t)), (x-\gamma(t))\>.
$$
Let $\Bc \in C^{\infty}(\Rn \times \Rb)$, then for $k \in \mathbb{N}$ define 
\begin{equation}\label{eq:gaussianbeam}
u_k(x,t) = k^{-1+\frac{n}{4}} \Bc(x,t) e^{ik \psi(x,t)}.
\end{equation}
Such a function $u_k$, is called a Gaussian beam. Given a geodesic $\gamma$, a function of this form can be constructed which is a quasi-solution of the wave equation. 
\begin{lemma}\label{l:Ralston}\cite{Ralston1982} Fix $T>0, t_0 \geq 0,$ and $(x_0, \xi_0) \in S^* \Rn$.  Let $\gamma(t)$ be the unit-speed geodesic with $(\gamma(t_0), \gamma'(t_0)) = (x_0, \xi_0)$. There exists $C>0$,  $\Bc \in \Ci(\Rn \times \Rb)$, and an $n \times n$ symmetric matrix-valued function $t \mapsto \Mc(t)$, so that for $u_k$ as in \eqref{eq:gaussianbeam} 
	\begin{enumerate}
		\item The $u_k$ are quasi-solutions of the wave equation
		\begin{equation}
			\sup_{t \in  [t_0,t_0+T]} \ltwo{ \p_t^2 u_k(\cdot, t) - \Delta_g u_k(\cdot, t) } \leq Ck^{-\frac{1}{2}}.
		\end{equation}
		\item The $u_k$ have small $L^2$ norm
		\begin{equation}
			\sup_{t \in [t_0, t_0+T]} \ltwo{u_k(\cdot, t)} \leq C k^{-1}.
		\end{equation}
		\item The energy of $u_k$ converges to $1$ uniformly for $t \in [t_0,t_0+T]$ as $k \ra \infty$
		\begin{equation}
			\sup_{t \in [t_0, t_0+T]}|E(u_k, t)-1| \leq \frac{CT}{k^{\frac{1}{2}}}.
		\end{equation}
		%This is a consequence of standard energy identity for the wave equation with inhomogeneity. See also Ralston (5.0)
	\end{enumerate}
%\begin{align}
%&\sup_{t \in  [t_0,t_0+T]} \ltwo{ \p_t^2 u_k(\cdot, t) - \Delta_g u_k(\cdot, t) } \leq Ck^{-\frac{1}{2}} \text{ for } k \geq 1, \label{eq:gaussianquasi} \\
%&\sup_{t \in [t_0, t_0+T]} \ltwo{u_k(\cdot, t)} \leq C k^{-1} \label{eq:Ltwouk}\\
%&\sup_{t \in [t_0, t_0+T]}|E(u_k, t)-1| \leq \frac{CT}{k^{\frac{1}{2}}}.\label{eq:Euknorm}
%\end{align}
\end{lemma}
%By the last assertion, it can be assumed that $\lim_{k \ra \infty} E(u_k,t)=1$ for all $t \in [t_0,t_0+T]$. 
Using coordinate charts and a partition of unity, this construction can be extended to manifolds, which results in a sequence $\{u_k\} \subset C^{\infty}(M \times \Rb)$ such that the appropriate analogues of statements 1), 2) and 3) hold. 

Before using these Gaussian beams to construct quasi-solutions of \eqref{TDWE}, some necessary preliminaries are introduced. 

First, at points in the construction of the quasi-solutions of \eqref{TDWE} we will require an $L^{\infty}$ bound on $\p_t W$, $\nabla W$, and $\nabla^2 W$. Because $W$ is only $C^0$ we replace $W$ by a smooth function $W_j$ which approximates $W$ in $L^{\infty}$. In particular, let $\{W_j\}_{j=1}^{\infty} \in C^{\infty}(M \times [0,\infty))$ be a sequence of nonnegative functions, such that 
\begin{equation}\label{eq:WjtoW}
	\nm{W_j-W}_{L^{\infty}(M \times [0,\infty))} < \frac{1}{j}.
\end{equation}
Such a sequence is guaranteed to exist by the density of $C^{\infty}(M \times [0,\infty))$ in $C^0(M \times [0,\infty))$. 

Now we define functions $G(\gamma, t_0,t)$, resp. $G_j(\gamma, t_0,t)$, which exponentially decrease according to the amount of damping $W$, resp. $W_j$, encountered along the geodesic $\gamma$ from $t_0$ to $t$. Given $(x_0, \xi_0) \in S^*M$, $t_0 \geq 0$, and $\gamma(s)$ the unit-speed geodesic with $(\gamma(t_0), \gamma'(t_0))=(x_0, \xi_0)$, define
\begin{align}
G(\gamma,t_0,t) &= \exp\left( - \int_{t_0}^{t} W(\gamma(s), s) ds \right), \label{propagatordef} \\
G_j(\gamma,t_0,t) & =\exp\left(- \int_{t_0}^{t} W_j(\gamma(s), s) ds \right) \label{propagatordef1}.
\end{align}

As a final preliminary, we point out that the $G$'s are always bounded by 1 and that the convergence of $W_j$ to $W$ implies that $G_j$ converges to $G$ as well. 
\begin{lemma}\label{l:Gbounded}
For $G$ and $G_j$ as defined above
\begin{enumerate}
	\item 
	\begin{equation}
		|G(\gamma,t_0,t)| \leq 1,\qquad |G_j(\gamma,t_0,t)| \leq 1.
	\end{equation}
	\item 
	For any $T\geq 0$, there exists $J \in \mathbb{N}$, such that for $j \geq J$ and any $t_0 \geq 0$
	\begin{equation}
		\sup_{t \in [t_0,t_0+T]} |G(\gamma,t_0,t)^2- G_j(\gamma,t_0,t)^2| \leq \frac{4T}{j}.
	\end{equation}
\end{enumerate}

\end{lemma}
\begin{proof}
The bounds in part 1 follow immediately from the definition of $G$ and $G_j$. \\
To see the convergence in part 2
\begin{align}
|G&(\gamma,t_0,t)^2- G_j(\gamma,t_0,t)^2|\\
&= \left|\exp\left(- 2\int_{t_0}^{t} W(\gamma(s),s) ds\right) - \exp\left(-2\int_{t_0}^{t} W_j(\gamma(s),s) ds\right) \right|\\
&=\exp\left(-2\int_{t_0}^{t} W(\gamma(s),s) ds\right) \left| 1 - \exp\left(-2\int_{t_0}^{t} W_j(\gamma(s), s) - W(\gamma(s),s) ds \right) \right|. 
\end{align}
Now note that the first factor $|\exp(-2\int W ds)| \leq 1$. Combining this with the convergence of $W_j$ to $W$ in \eqref{eq:WjtoW} we obtain 
\begin{align}
|G(\gamma,t_0,t)^2- G_j(\gamma,t_0,t)^2|&\leq \left|1-\exp\left(\frac{-2(t-t_0)}{j}\right)\right| \\
&\leq \frac{2(t-t_0)}{j} + O\left(\frac{(t-t_0)^2}{j^2}\right) \leq \frac{4(t-t_0)}{j},
\end{align}
where the second inequality follows from Taylor's theorem for $j$ large enough, since $(t-t_0) \leq T$. Taking the supremum over $t \in [t_0, t_0+T]$ gives part 2. 
\end{proof}
Now, we define our prospective quasi-solution of the damped wave equation \eqref{TDWE} as  
\begin{equation}\label{vkdef}
v_{k,j}(x,t) = G_j(\gamma, t_0, t) u_k(x,t). 
\end{equation}
The intuitive idea is that solutions of \eqref{TDWE} decrease exponentially as they encounter the damping. Since $u_k$ is concentrated along the geodesic $\gamma$, to turn it into a quasi-solution of \eqref{TDWE} it should decay exponentially in proportion to the damping encountered along $\gamma$. 

The construction \eqref{vkdef} comes from \cite{Lebeau1996}, where $W$ is autonomous. When $W$ is autonomous, $G$ is constructed as the propagator of the defect measure for the damped wave equation, and was used in \cite{Klein2017, KeelerKleinhenz2023} to prove analogous statements for matrix valued and pseudodifferential damping respectively. Although we do not consider defect measures here, we are still able to make use of this construction. The key differences here are allowing $G$ to depend on the start time $t_0$, and accounting for the time-dependence of $W$.

We now show that this $v_{k,j}$ is a quasi-solution of \eqref{TDWE} with damping $W_j$, and as $k \ra \infty$ its energy approaches $G_j(\gamma, t_0,t)^2$. 
\begin{proposition}\label{quasiprop}
Given $(x_0, \xi_0) \in S^* M$, $t_0 \geq 0$, $j \in \Nb$ and $T >0$, let $u_k(t,x)$ be as in \eqref{eq:gaussianbeam} and $v_{k,j}$ as in \eqref{vkdef}. There exists a constant $C=C(T,j,t_0) >0,$ so that 
\begin{enumerate}
	\item The $v_{k,j}$ are quasi-solutions of \eqref{TDWE} with damping $W_j$
	\begin{equation}
		\sup_{t \in [t_0,t_0+T]} \ltwo{(\p_t^2 - \Delta_g +2W_j  \p_t) v_{k,j}(\cdot, t) } \leq C k^{-\frac{1}{2}}.
	\end{equation}
	\item The energy of $v_{k,j}$ converges to $G_j(\gamma,t_0,t)^2$ uniformly for $t \in [t_0,t_0+T]$ as $k \ra \infty$
	\begin{equation}
		 \sup_{t \in [t_0,t_0+T]} | E(v_{k,j},t) - G_j(\gamma, t_0, t)^2| \leq \frac{CT}{k^{\frac{1}{2}}}.
	\end{equation}
\end{enumerate}
\end{proposition}
\begin{proof}
First, we directly compute
\begin{align}
(\p_t^2 -\Delta_g +2W_j(x,t) \p_t) G_j(\gamma, t_0,t) u_k(x,t)=& G_j (\p_t^2 -\Delta_g) u_k\\
&+ 2 \p_t G_j \p_t u_k + (\p_t^2 G_j) u_k \\
&+ 2 W_j (\p_t G_j) u_k+ 2W_j G_j \p_t u_k.
%(\p_t^2 -\Delta_g +& 2W(x,t+t_0) \p_t) G(x_0,\xi_0,t_0) u_k(x,t+t_0) = G (\p_t^2 -\Delta_g) u_k(x,t+t_0)\\
%&+ 2 \p_t G \p_t u_k(x,t+t_0) + (\p_t^2 G) u_k(x,t+t_0) \\
%&+ 2 W(x,t+t_0) \p_t G u_k(x,t+t_0) + 2W(x,t+t_0) G \p_t u_k(x,t+t_0).
\end{align}
Now since $\p_t G_j(\gamma,t_0, t)=-W_j(\gamma(t), t)) G_j(\gamma,t_0,t)$, after regrouping terms
\begin{align}
(\p_t^2 -\Delta_g +2W_j \p_t) G_j u_k=& G_j (\p_t^2 -\Delta_g) u_k  \\
&+ \bigg(-\p_t W_j(\gamma(t),t) +W_j^2(\gamma(t),t)  -2W_j (x,t) W_j (\gamma(t), t)\bigg) G_j u_k  \\
&+ 2\bigg(W_j (x,t)-W_j(\gamma(t), t) \bigg) G_j \p_t u_k\label{quasiest}.
\end{align}
The three terms on the right hand side will be estimated in turn. First, the $u_k$ are quasi-solutions of the wave equation by Lemma \ref{l:Ralston} part 1, and $G_j$ is bounded by Lemma \ref{l:Gbounded}, so
\begin{equation}\label{eq:Gjukquasi}
\sup_{t \in [t_0,t_0+T]} \ltwo{G_j (\p_t^2 -\Delta_g) u_k(\cdot, t) } \leq C k^{-\frac{1}{2}}. 
\end{equation}
Second, note that $W_j \in C^{\infty}(M \times [t_0,t_0+T])$ and $M \times [t_0,t_0+T]$ is compact. Therefore $W_j$ and $\p_t W_j$ are bounded on $M \times [t_0,t_0+T]$. Note that this is one of the steps which required us to replace $W$ by $W_j$, as $\p_t W$ need not exist. The boundedness of $W_j$ and $\p_t W_j$, along with the boundedness of $G_j$ from Lemma \ref{l:Gbounded}, means there exists $C>0$ depending on $T,j$ and $t_0,$ such that
\begin{align}
\sup_{t \in [t_0,t_0+T]}& \ltwo{\left(-\p_t W_j(\gamma(t),t)+W_j^2(\gamma(t),t) -2W_j(\cdot,t) W_j(\gamma(t),t)\right) G_j u_k(\cdot,t) } \\
&\leq C \sup_{t\in [t_0,t_0+T]} \ltwo{u_k(\cdot, t)}  \leq C k^{-1}, \label{eq:vkjerror1}
\end{align}
where the final equality follows from Lemma \ref{l:Ralston} part 2. 

To estimate the final term on the right hand side of \eqref{quasiest}, at each time $t$, Taylor expand $W_j$ in the $x$ variable around $\gamma(t)$ 
\begin{equation}
W_j(x,t) = W_j(\gamma(t), t) + (x-\gamma(t)) \cdot \nabla_x W_j(\gamma(t), t) + (x-\gamma(t))\cdot C_{2} \cdot (x-\gamma(t)),
\end{equation}
where $|C_2| \leq \nm{\nabla^2 W_j}_{L^\infty(M \times[t_0,t_0+T])}$. Note this Taylor series expansion is one of the steps which required us to replace $W$ by $W_j$ as $\nabla W$ and $\nabla^2 W$ need not exist. We now have 
\begin{align}
2\bigg(W_j(x,t)-W_j(\gamma(t), t)\bigg) G_j \p_t u_k(x,t)= \bigg(& 2(x-\gamma(t)) \cdot \nabla_x W_j(\gamma(t), t)  \\
&+ (x-\gamma(t)) \cdot C_2 \cdot (x-\gamma(t)) \bigg) G_j \p_t u_k(x,t).
\end{align}
Taking the $L^2$ norm squared of both sides 
\begin{align}
&\ltwo{\bigg(W_j(\cdot,t)-W_j(\gamma(t), t)\bigg) G_j \p_t u_k(\cdot,t) }^2 \\
&\leq \int_M \bigg(|2(x-\gamma(t))\cdot \nabla_x W_j| + |(x-\gamma(t))\cdot C_2 \cdot (x-\gamma(t))| \bigg)^2 \,|G_j \p_t u_k|^2 dx_g.
\end{align}
Using coordinate charts and a partition of unity on $M$, without loss of generality we can assume that the integrand on the right hand side is contained in a single coordinate chart. Using this, the definition of the $u_k$ in \eqref{eq:gaussianbeam}, and letting $C_1 = 2\nm{\nabla W_j}_{L^{\infty}(M \times [t_0,t_0+T])}$
\begin{align}
&\ltwo{\bigg(W_j(\cdot,t)-W_j(\gamma(t), t)\bigg) G_j \p_t u_k(\cdot,t) }^2\\
&\leq \intrn  \bigg( C_{1} |x-\gamma(t)|+|C_2| |x-\gamma(t)|^2 \bigg)^2 \, \bigg|k^{\frac{n}{4}} e^{ik\<x-\gamma(t), \gamma'(t) \>} e^{\frac{ik}{2} \<\Mc(x-\gamma(t)), \, x-\gamma(t)\>} \mathcal{B}(x,t) \bigg|^2 dx \\
&\leq  C k^{\frac{n}{2}} \intrn \bigg(C_{1} |x-\gamma(t)| + |C_2| |x-\gamma(t)|^2 \bigg)^2 \, \left|e^{\frac{ik}{2}(\<\Mc(x-\gamma(t)), \, x-\gamma(t)\>}\right|^2 dx.
\end{align}
Taking the change of variables $k^{1/2} (x-\gamma(t)) = y, dx = k^{-\frac{n}{2}} dy,$ this becomes
\begin{equation}
	\ltwo{2\left(W_j(\cdot,t)-W_j(\gamma(t), t)\right) G_j \p_t u_k(\cdot, t)}^2 \leq  C k^{-1} \intrn |y e^{\frac{i}{2} \<\Mc y,y\>} |^2 dy  \leq C k^{-1},
\end{equation}
where $C$ depends on $T, j$ and $t_0$, and the final integral is finite by the definition of $\Mc$. Finally, taking square roots and then the supremum over $t \in [t_0,t_0+T],$ 
\begin{equation}
	\sup_{t \in [t_0,t_0+T]} \ltwo{(W_j(\cdot,t)-W_j(\gamma(t),t))G_j\p_t u_k(\cdot,t)} \leq Ck^{-1/2}.
\end{equation}
Combining this with \eqref{quasiest}, \eqref{eq:Gjukquasi}, and \eqref{eq:vkjerror1}  gives the desired inequality completing the proof of part 1. 

To see part 2, first compute 
\begin{equation}
	E(v_{k,j},t) = \frac{1}{2} \int_M |G_j(\gamma, t_0, t) \nabla u_k |^2 + |G_j(\gamma, t_0, t) \p_t u_k + \p_t G_j(\gamma, t_0, t) u_k|^2 dx_g. 
\end{equation}
Now note that by the $L^2$ size of $u_k$ from Lemma \ref{l:Ralston} part 2 
\begin{equation}
	\frac{1}{2}\int_M |\p_t G_j(\gamma, t_0, t) u_k|^2 dx_g = |W_j(\gamma(t),t) G_j(\gamma, t_0,t)|^2 \int_M |u_k(x,t)|^2 dx_g \leq C k^{-2}.
\end{equation}
Thus
\begin{align}
	|E(v_{k,j},t)-G_j(\gamma,t_0,t)^2| &\leq  G_j(\gamma, t_0, t)^2 \left|\frac{1}{2}\int_M |\nabla u_k(x,t)|^2 + |\p_t u_k(x,t)|^2 dx_g-1\right|+Ck^{-2} \\
	&\leq \left|E(u_k,t)-1\right| + Ck^{-2} \leq \frac{CT}{k^{\frac{1}{2}}},
\end{align}
where the second inequality follows by the boundedness of $G_j$ from Lemma \ref{l:Gbounded}, and the final inequality follows for $k$ large enough from the uniform convergence of $E(u_k,t) \ra 1$ as $k \ra \infty$ via Lemma \ref{l:Ralston} part 3. 
\end{proof}

We now show that for any geodesic $\gamma$, starting time $t_0$, and duration $T$, there are exact solutions of the damped wave equation \eqref{TDWE} with energy arbitrarily close to $G(\gamma,t_0,t)^2$ for $t \in [t_0, t_0+T]$. This is done by considering $\omega_k$ solutions of \eqref{TDWE} with initial data given by $v_{k,J}$, the quasi-solution defined in Proposition \ref{quasiprop}, for some large fixed $J$. We then show that for $k$ large enough, $E(\omega_k)$ is close to $E(v_{k,J})$. Then we note $E(v_{k,J})$ is close to $G_J^2(\gamma,t_0,t)$ by Proposition \ref{quasiprop}, which in turn is close to $G^2(\gamma,t_0,t)$ by Lemma \ref{l:Gbounded}.
\begin{proposition}\label{energyapproxprop} Suppose $W \in C^0(M \times [0,\infty))$. 
For all $T, \e>0, t_0 \geq 0$ and any unit-speed geodesic $\gamma(t)$, there exists an exact solution $u$ of \eqref{TDWE} with
\begin{enumerate}
	\item Initial energy close to 1
		\begin{equation}
			|E(u,t_0)-1|<\e.
		\end{equation}
	\item The energy is close to $G(\gamma,t_0,t)^2$, that is for any $t \in [t_0,t_0+T]$
	\begin{equation}
		|E(u,t) - G(\gamma,t_0,t)^2 | < \e. 
	\end{equation}
	\item Furthermore, there is a lower bound on the energy in terms of $G(\gamma, t_0,t)$ and the initial energy. That is for all $t \in [t_0,t_0+T]$
	\begin{equation}
		E(u,t) > E(u,t_0) (G(\gamma,t_0,t)^2-2\e).
	\end{equation}
\end{enumerate}
\end{proposition}
\begin{proof}
By Lemma \ref{l:Gbounded} and the convergence of the $W_j$ to $W$ in \eqref{eq:WjtoW} there exists $J$ large enough so that 
\begin{align}
	&\sup_{t \in [t_0,t_0+T]} |G_J(\gamma,t_0,t)^2-G(\gamma,t_0,t)^2| < \frac{\e}{3},\label{eq:GJclosetoG}\\
	&\nm{W-W_J}_{L^{\infty}(M\times [0,\infty))} < \frac{\e^2}{6000T}.
\end{align}
Let $v_{k,J}$ be as defined in \eqref{vkdef} with $j=J$. Then define $\omega_k$ as the unique solution of 
$$
\begin{cases} (\p_t^2 -\Delta + 2W \p_t) \omega_k =0 \\
\omega_k(x,t_0)= v_{k,J}(x,t_0)\\
\p_t \omega_k(x,t_0) = \p_t v_{k,J}(x,t_0).
\end{cases}
$$ 
The idea is to now choose $k$ large enough so that $\omega_k(x,t)$ is the desired solution. 

1) It is immediate from $G_j(\gamma, t_0,t_0)=1$ and Proposition \ref{quasiprop} that
\begin{equation}\label{eq:Eomegak}
E(\omega_k, t_0) = E(v_{k,J}, t_0) \ra 1 \text{ as } k \ra \infty. 
\end{equation}
So for $k$ large enough, setting $u(x,t)=\omega_k(x,t)$ satisfies part 1.

2) To prove $\omega_k$ satisfies part 2 for $k$ large enough, we begin with the triangle inequality
\begin{align}
	|E(\omega_k, t) - G(\gamma,t_0,t)^2| \leq &|E(\omega_k, t) - E(v_{k,J},t)| + |E(v_{k,J},t)-G_J(\gamma, t_0,t)^2| \\
	&+ |G_J(\gamma,t_0,t)^2-G(\gamma,t_0,t)^2|.
\end{align}
By \eqref{eq:GJclosetoG} the last term is less than $\e/3,$ and by Proposition \ref{quasiprop} for $k$ large enough
\begin{equation}
	\sup_{t \in [t_0, t_0+T]} |E(v_{k,J},t) - G_J(\gamma,t_0,t)^2| < \frac{\e}{3}.
\end{equation}
Thus to show part 2, it is enough to show for $k$ large enough that 
\begin{equation}
	\sup_{t \in [t_0, t_0+T]} |E(\omega_k,t)-E(v_{k,J},t)| < \frac{\e}{3}.
\end{equation}
%To see the first note
%\begin{equation}
%E(v_{k,J}, t) = \frac{1}{2} \int_{\Omega} |G_J\p_t u_k- W_J(\gamma(t), t) G_J u_k|^2 + |G_J \nabla u|^2 dx_g. 
%\end{equation}
%Since $W_J$ and $G_J$ are bounded, and by construction of $u_k$,
%$$
%\int_{\Omega} | W_J(\gamma(t), t) G_J u_k|^2 dx_g \leq C \ltwo{u_k(\cdot, t)}^2 \leq C k^{-2}.
%$$
%And so 
%$$
%\lim_{k\ra \infty} E(v_{k,J}, t) = \lim_{k \ra \infty} \frac{1}{2} \int |G_J \p_t u_k|^2 + |G_J \nabla u_k|^2 dx_g = |G_J|^2 \lim_{k \ra \infty} E(u_k, t) = G_J(\gamma,t_0,t)^2,
%$$
%where in the final equality the fact that $\limk E(u_k, t)=1$ for $t \in [t_0,t_0+T]$ was used.  
To see this, begin again with the triangle inequality 
\begin{equation}
|E(\omega_k, t)^{1/2} - E(v_{k,J}, t)^{1/2} | \leq E(\omega_k -v_{k,J}, t)^{1/2}.
\end{equation}
Note $E(\omega_k -v_{k,J}, t_0)=0$ and we will control the time derivative of $E(\omega_k-v_{k,J},t)$.
To do so, first let $f_k = (\p_t^2 -\Delta + 2W_J \p_t )v_{k,J}$. Then  
\begin{align}
(\p_t^2 -\Delta +2W \p_t) (\omega_k - v_{k,J}) &= f_k + 2(W - W_J) \p_t v_{k,J}\\
&=f_k + 2(W - W_J) (- W_J(\gamma(t), t)G_J u_k+G_J\p_t u_k).
\end{align}
%So by Proposition \ref{quasiprop}, for any $\e, T>0$ there exists a $C_{\e,T}>0$ such that 
%$$
%\sup_{t \in [0,T]} \ltwo{f_k(\cdot, t)} \leq C_{\e,T} k^{-\frac{1}{2}}.
%$$
Now taking the time derivative of $E(\omega_k-v_{k,J},t)$
\begin{align}
\p_t E(\omega_k -v_{k,J} t) = \frac{1}{2} \int_M \bigg((\p_t^2 &-\Delta_g) (\omega_k-v_{k,J}) \overline{\p_t (\omega_k - v_{k,J})}\\
&+ (\p_t^2 -\Delta_g)\overline{(\omega_k-v_{k,J})} \p_t(\omega_k-v_{k,J}) \bigg) dx_g \\
=\Re \int_M \bigg( f_k &+2(W-W_J) (-W_J G_J u_k+G_J \p_t u_k)\\
&-2W\p_t(\omega_k -v_{k,J}) \bigg) \p_t \overline{(\omega_k-v_{k,J})} dx_g.
\end{align}
Note that since $W \geq 0$, the final term on the right hand side 
\begin{align}
	\int_M -2 W |\p_t (\omega_k - v_{k,J})|^2 dx_g \leq 0.
\end{align} 
Combining this with the H\"older inequality and the boundedness of $G_J$, for all $t \in [t_0, t_0+T]$
\begin{align}
	|\p_t E(\omega_k-v_{k,J},t)| \leq \sup_{t \in [t_0, t_0+T]}\bigg(&\ltwo{f_k} + 2 \lp{W-W_J}{\infty} \lp{W_J}{\infty} \ltwo{u_k} \\
	& +2 \lp{W-W_J}{\infty} \ltwo{\p_t u_k}\bigg) \ltwo{\p_t (\omega_k - v_{k,J})}. \label{eq:omegakIntermediate}
\end{align}
To control the remaining terms, first note: since $\omega_k$ is a solution of \eqref{TDWE} its energy is non-increasing, and by the convergence of $E(v_{k,J},t)$ to $G_J(\gamma,t_0,t)$ from Proposition \ref{quasiprop}, for $k$ large enough
\begin{align}
\ltwo{\p_t(\omega_k-v_{k,J})} &\leq 2 E(\omega_k, t) + 2 E(v_{k,J}, t) \\
&\leq 2 E(\omega_k, t_0) + 2 |G_J(\gamma,t_0,t)|^2 + \frac{1}{2} \\
&\leq 2+\frac{1}{2}+ 2|G_J(\gamma, t_0, t)|^2+\frac{1}{2} \leq 5,
\end{align}
where the third inequality follows by the convergence of $E(\omega_k,t_0)$ to 1 from \eqref{eq:Eomegak}, and the final inequality follows from the boundedness of $G_j$ in Lemma \ref{l:Gbounded}. Note also by Lemma \ref{l:Ralston} part 3, that 
\begin{align}
&\ltwo{\p_t u_k} \leq 2 E(u_k,t) = 2 E(u_k,0) \leq 3.
\end{align}
Combining these with \eqref{eq:omegakIntermediate}, then controlling $f_k$ with Proposition \ref{quasiprop}, $u_k$ with Lemma \ref{l:Ralston} part 2, and $W-W_J$ by our assumption on $J$
\begin{align}
	\sup_{t \in [t_0,t_0+T]} |\p_t E(\omega_k-v_{k,J},t)|&\leq  \sup_{t \in [t_0, t_0+T]} C \ltwo{f_k} + 30 \lp{W-W_J}{\infty} +C \ltwo{u_k} \\
	& \leq C k^{-1/2} + \frac{\e^2}{200T}
\end{align}
Now choosing $k$ large enough so that $Ck^{-1/2} \leq \frac{\e^2}{200T}$ we have 
\begin{equation}
	\sup_{t \in [t_0,t_0+T]} |\p_t E(\omega_k-v_{k,J},t)| \leq \frac{\e^2}{100T}.
\end{equation}
Thus integrating the time derivative and using $E(\omega_k- v_{k,J},t_0)=0$
\begin{equation}
	0\leq \sup_{t \in [t_0,t_0+T]} E(\omega_k- v_{k,J},t) \leq \frac{\e^2}{100}.
\end{equation}
Therefore for $t \in [t_0, t_0+T]$
\begin{equation}
	|E(\omega_k, t)^{1/2} - E(v_{k,J})^{1/2}| \leq E(\omega_k- v_{k,J},t)^{1/2} \leq \frac{\e}{10}.
\end{equation}
Thus for any $t \in [t_0,t_0+T]$ and for $k$ large enough, again using Proposition \ref{quasiprop} to estimate $E(v_{k,J},t)$ in terms of $G_J(\gamma,t_0,t)$ and using that $E(\omega_k,t)=E(\omega_k,t_0)$
\begin{align}
	|E(\omega_k, t)-E(v_{k,J},t)| &= |E(\omega_k, t)^{1/2}-E(v_{k,J},t)^{1/2}||E(\omega_k, t)^{1/2}+E(v_{k,J},t)^{1/2}|\\
	&\leq \frac{\e}{10}(E(\omega_k, t_0)^{1/2} + |G_J(\gamma, t_0,t)|+\frac{1}{2})\\
	& \leq \frac{\e}{10}\left(1+\frac{1}{2}+|G_J(\gamma,t_0,t|)+\frac{1}{2}\right) \leq \frac{3\e}{10}<\e/3, 
\end{align}
where the third inequality uses the convergence of $E(\omega_k,t_0)$ to 1 from \eqref{eq:Eomegak}, and the boundedness of $G_j$ in Lemma \ref{l:Gbounded}.
%
%
%since $E(u_k,0) \leq \frac{4}{3}$ and $E(v_{k,J},t) \leq 2 E(u_k,0)$ for $k$ large enough. This, along with the control over $f_k$ from Proposition \ref{quasiprop}, $\lp{W-W_J}{\infty} < \frac{\d}{100T}$, and $\ltwo{u_k} \leq Ck^{-1}$, gives
%\begin{align}
%\sup_{t \in [t_0,t_0+T]} |\p_t E(\omega_k -v_k, t)| &\leq \sup_{t \in [t_0,t_0+T]} \int_M |f_k| |\p_t (\omega_k - v_{k,J})| + 2 \nm{W-W_J}_{L^{\infty}} G_J |\p_t u_k| |\p_t(\omega_k - v_{k,J})| \\
%&+ 2 \lp{W-W_J}{\infty} W_J G_J |u_k| |\p_t (\omega_k - v_{k,J})| dx_g \\
%&\leq \sup_{t \in [t_0,t_0+T]} \ltwo{f_k} \ltwo{\p_t (\omega_k-v_{k,J})} \\
%&+ 2\lp{W-W_J}{\infty} \ltwo{\p_t u_k} \ltwo{\p_t(\omega_k - v_{k,J})} \\
%&+  2 \lp{W-W_J}{\infty} \nm{W_J}_{L^{\infty}} \ltwo{u_k} \ltwo{\p_t (\omega_k-v_{k,J})} \\
%&\leq \sup_{t \in [t_0,t_0+T]} 4 \ltwo{f_k} + 16 \lp{W-W_J}{\infty} + C k^{-1}\\
%&\leq C k^{-1/2} + \frac{\d}{6T}.
%\end{align}
%Then choose $k$ large enough so that $Ck^{-1/2} \leq \frac{\d}{6T}$, so that $\sup_{t \in [t_0,t_0+T]} |\p_t E(\omega_k-v_{k,J}, t)| \leq \frac{\d}{3T}$. Since $E(\omega_k-v_{k,J},t_0)=0,$ integrating this time derivative for $t \in [t_0,t_0+T]$
%$$
%E(\omega_k-v_{k,J}, t) \leq \frac{\d}{3}.
%$$
%Therefore 
%$$
%|E(\omega_k,t)^{1/2} - E(v_{k,J},t)^{1/2}| \leq E(\omega_k- v_k, t)^{1/2} \leq \left(\frac{\d}{3}\right)^{1/2},
%$$
%and so $|E(\omega_k,t)-E(v_{k,J},t)| < 4 \left( \frac{\d}{3} \right)^{1/2}=\frac{\e}{3},$ for $k$ large enough. 
Thus part 2 holds. 

3) Now we show that $u=\omega_k$ satisfies part 3 for $k$ large enough. If we take $k$ large enough so that the first two parts hold, then for any $t \in [t_0,t_0+T],$ since $1>E(\omega_k,t_0)-\e$, $E(\omega_k,t) > G(\gamma,t_0,t)^2 -\e$, and the energy of solutions to \eqref{TDWE} is non-increasing
\begin{align}
E(\omega_k,t) &> E(\omega_k,t) (E(\omega_k,t_0) - \e) =E(\omega_k,t)E(\omega_k,t_0)-\e E(\omega_k, t)\\
&>E(\omega_k,t_0) ( G(\gamma,t_0,t)^2 - \e ) - \e E(\omega_k,t) \\
&> E(\omega_k,t_0) (G(\gamma,t_0,t)^2 -\e) - \e E(\omega_k,t_0) \\
&>E(\omega_k,t_0)(G(\gamma,t_0,t)^2 -2\e).
\end{align}
This is exactly the desired inequality.
\end{proof}
\subsection{Bound on Decay Rates: Proof of Theorem \ref{lowerboundthm}}
Now to prove the lower bounds on uniform stabilization rates, we apply Proposition \ref{energyapproxprop} part 3, to obtain a lower bound in terms of $G(\gamma, t_0,t)$. We then relate $L_{\infty}$ and $\Sigma(t)$ to $G(\gamma, t_0,t)$. Because Proposition \ref{energyapproxprop} only applies on a finite time interval $[t_0, t_0+T]$, we must argue by contradiction. 
\begin{proof}
1) As a first step, it is necessary to see that $t L(t)$ is super additive. This is the step that requires taking the infimum over starting times $t_0$ in the definition of $L(t)$. Let $t, r \geq 0$
\begin{align}
(t+r) L(t+r) &= \inf_{t_0 \geq 0, \gamma} \int_{t_0}^{t_0+r+t} W(\gamma(s), s) ds \\
&= \inf_{t_0 \geq 0, \gamma}  \left( \int_{t_0}^{t_0+r} W(\gamma(s), s) ds + \int_{t_0+r}^{t_0+r+t} W(\gamma(s), s) ds \right)\\
& \geq \inf_{t_0 \geq 0, \gamma} \int_{t_0}^{t_0+r} W(\gamma(s), s) ds + \inf_{t_0 \geq 0, \gamma}  \int_{t_0+r}^{t_0+r+t} W(\gamma(s),s) ds \\
&= \inf_{t_0 \geq 0, \gamma}  \int_{t_0}^{t_0+r} W(\gamma(s), s)ds + \inf_{t_0 \geq 0, \gamma}  \int_{t_0}^{t_0+t} W(\gamma(s), s) ds \\
&= tL(t) + rL(r). 
\end{align}
Therefore by Fekete's Lemma, $L_{\infty} = \lim_{t \ra \infty} L(t) = \sup_{t \in [0,\infty)} L(t)$ and thus $L(t) \leq L_{\infty}$ for all $t$. 

Now it can be shown that $\alpha \leq 2 L_{\infty}$.  Assume otherwise, so $\alpha=2 L_{\infty} + 3 \eta$ for some $\eta >0$, and a contradiction will be obtained by applying Proposition \ref{energyapproxprop} part 3. Since $2(L_{\infty} + \eta) < \alpha$, there exists $C^*>0,$ such that for all $t, t_0 \geq 0$, and all solutions $u$ of \eqref{TDWE}
\begin{equation}\label{edecaycontra}
E(u,t+t_0) \leq C^* E(u,t_0) e^{-2t(L_{\infty}+\eta)}.
\end{equation}
This right hand side will be manipulated to be bounded from above by $G(\gamma, t_0, T+t_0)^2-\d$ at some time $t=T$ and for some $\d>0$. First, choose $T>0$ large enough, so that $\max(C^*,1) < e^{T \eta},$ then
\begin{equation}
C^* e^{-2T(L_{\infty}+\eta)} < e^{-2TL_{\infty}-T\eta}.
\end{equation}
Now by subadditivity $L(T) \leq L_{\infty}$ for all $T$, so
\begin{equation}\label{Linfeq}
C^* e^{-2T(L_{\infty}+\eta)} < e^{-2TL_{\infty} - T\eta} \leq e^{-2TL(T)- T \eta}.
\end{equation}
Now note that 
\begin{equation}
L(T)=-\frac{1}{T} \sup_{t_0 \geq 0, \gamma} \ln(G(\gamma, t_0, T+t_0)).
\end{equation}
Thus there exists a geodesic $\gamma$ and a starting time $t_0 \in \Rb,$ such that 
\begin{equation}
	\ln G(\gamma,t_0,T+t_0) > -TL(T) - \frac{T \eta}{2}.
\end{equation}
Combining this with \eqref{Linfeq}
\begin{equation}\label{GLTeq}
C^* e^{-2T(L_{\infty}+\eta)}<e^{-2TL(T)-T \eta} < G(\gamma,t_0,T+t_0)^2.
\end{equation}
Since these are inequalities between constant real numbers, there exists $\d>0$, such that 
\begin{equation}
C^* e^{-2T(L_{\infty}+\eta)} < G(\gamma, t_0, T+t_0)^2 - \d.
\end{equation}
Now by Proposition \ref{energyapproxprop} part 3, there exists an exact solution $u$ of \eqref{TDWE} such that 
\begin{equation}
E(u,T+t_0) > E(u,t_0)(G(\gamma, t_0, T+t_0)^2-\d) > C^* E(u,t_0)e^{-2T(L_{\infty}+\eta)},
\end{equation}
which contradicts \eqref{edecaycontra}. Therefore, $\alpha \leq 2L_{\infty}$.

2) The proof of 2) is analogous to that of 1). For completeness, the details are included. To begin, we claim that for any $\d, T>0$ there exists a solution $u$ to \eqref{TDWE} such that 
\begin{equation}\label{eq:ElowSigma}
E(u,T) > E(u,0) \exp(-2\Sigma(T)-\d).
\end{equation}
To see this, fix $\d, T>0$ and note  
\begin{equation}
\Sigma(T) = \inf_{\gamma} \int_0^T W(\gamma(s),s)ds= - \sup_{\gamma} \ln(G(\gamma,0,T)).
\end{equation}
Thus there exists a geodesic $\gamma$, such that 
\begin{equation}
	-\Sigma(T) - \frac{\d}{2} < \ln(G(\gamma,0,T)).
\end{equation}
Therefore 
\begin{equation}
	\exp(-2\Sigma(T)- \d) < G(\gamma,0,T)^2.
\end{equation}
Now because the above inequality is between constant real numbers, there exists an $\e>0$ small enough, so that
\begin{equation}
	 \exp(-2\Sigma(T)-\d) < G(\gamma,0,T)^2- 2\e.
\end{equation}
Then by Proposition \ref{energyapproxprop} part 3, there exists $u$ a solution of \eqref{TDWE} such that 
\begin{align}
E(u,T) &> E(u,0) (G(\gamma,0,T)^2 - 2\e) \\
&> E(u,0) \exp(-2\Sigma(T) - \d).
\end{align}
This completes the proof of the claim \eqref{eq:ElowSigma}. 

Now it can be shown that $\sigma \leq 2$. Assume otherwise, so there exists $\eta>0$ such that $\sigma = 2 + 3 \eta$. Then since $2 + 2\eta < \sigma$, there exists $C^*>0,$ such that for all $t \geq 0$ and all solutions $u$ of \eqref{TDWE}
\begin{equation}
E(u,t) \leq C^* E(u,0) \exp(-(2+2\eta) \Sigma(t)).
\end{equation}
Since $\limt \Sigma(t) =\infty$, there exists $T>0$ large enough, so that $\exp(\eta \Sigma(T))>\max(C^*, e^{1/2})$. Then 
\begin{align}
E(u,T) &\leq C^* \exp(-(2+2\eta) \Sigma(T)) E(u,0) \\
&< \exp(-(2+\eta) \Sigma(T)) E(u,0)\\
&<\exp\left(-2\Sigma(T)-\frac{1}{2}\right) E(u,0).
\end{align}
But by \eqref{eq:ElowSigma} there exists $u$ solving \eqref{TDWE}, such that 
\begin{align}
E(u,T) &> E(u,0) \exp\left(-2\Sigma(T) - \frac{1}{2}\right) \\\
&>E(u,0) \exp(-(2+\eta) \Sigma(T)) \\
&> C^* E(u,0)  \exp(-(2+2\eta) \Sigma(T)) \\
&>E(u,T).
\end{align}
This is a contradiction and so we must have $\sigma \leq 2$. 

3) When $\Sigma(t) < K<\infty$ to see uniform stabilization cannot occur, assume otherwise. So suppose there exists $r(t) \ra 0$ such that $E(u,t) \leq r(t) E(u,0)$ for all $u$ solving \eqref{TDWE}. Then pick $T$ large enough so that $r(T) < \exp(-3K)$. Then
for all $u$ solving \eqref{TDWE}
$$
E(u,T) \leq r(T) E(u,0) < \exp(-3K) E(u,0).
$$
But by \eqref{eq:ElowSigma} for any $\d \in (0,K)$ there exists $u$ solving \eqref{TDWE} with 
\begin{align}
E(u,T)  &> E(u,0) \exp(-2\Sigma(T) -\d) \\
&> E(u,0) \exp(-2K  - \d) \\
&> E(u,0) \exp(-3K) > E(u,T),
\end{align}
which is a contradiction. 
\end{proof}
\begin{remark} When Assumption \ref{TGCC} is not satisfied it is more convenient to work with $\Sigma(t)$, as it distinguishes how $L(t)$ approaches $0$ as $t \ra \infty$. Note that an infimum is not taken over starting times $t_0$ in the definition of $\Sigma(t)$. This was necessary to show $tL(t)$ was sub-additive, but when Assumption \ref{TGCC} is not satisfied 
	\begin{equation}
		\inf_{t_0, \gamma} \int_{t_0}^{t_0+t} W(\gamma(t), t) dt = 0.
	\end{equation}
\end{remark}

\section{Proof of Wave Observability Inequalities}\label{observesection}
The proof of Proposition \ref{uniobserveprop} proceeds in three steps. First, following the standard approach for observability inequalities \cite{BardosLebeauRauch1992, BurqGerard1997, LRLTT}, we prove a weak version on a finite time interval $(0,T)$ using propagation of defect measure. Then, we eliminate an error term by showing there are no solutions ``invisible" to the observation set. Our proof showing there are no ``invisible" solutions is what allows us to consider more general initial data than \cite{LRLTT}. Finally we use a compactness argument to obtain control of the observability constant on time intervals $(t_0, t_0+T)$ independent of the starting time $t_0$. Obtaining this uniformity in $t_0$ is new and is what enables us to obtain decay rates for solutions of the damped wave equation \eqref{TDWEb}.

We begin with a finite time version of Assumption \ref{TGCCp}.
\begin{assumption}\label{ftGCC}(finite-time TGCC)
Fix $T>0$ and consider $Q,$ an open set in $\Omegab \times (0,T)$. Assume for all unit-speed generalized geodesics $\gamma,$ there exists $t \in (0,T),$ such that $(\gamma(t),t) \in Q$.
\end{assumption}

\begin{remark}
	\begin{enumerate}
		\item If $W \in C^0_u(\Omegab \times [0,\infty))$ satisfies the TGCC for some $T_0$, that is Assumption \ref{TGCCp}, then there exists $\d>0$, such that $\{W>\d\}$ satisfies Assumption \ref{ftGCC} on $\Omega \times [0,T_0]$.
		\item This is the time-dependent generalization of the GCC used in \cite[Definition 1.6]{LRLTT}, which implies exponential uniform stabilization for time-periodic damping \cite[Corollary 1.14]{LRLTT}, and a time-dependent observability inequality \cite[Theorem 1.8]{LRLTT}.
	\end{enumerate}

\end{remark}
%When the boundary conditions are Dirichlet the next result follows from Lemma 2.1 of \cite{LRLTT}, so throughout in the statement and proof it is assumed that the boundary condition is Neumann or the boundary is empty. In particular $H$ is always $H^1(\Omega)$ and so is written as such. Also define $H^{-1}(\Omega)$ as the dual of $H^1(\Omega)$. 
We now state and prove an observability inequality with an error. The proof is analogous to that of \cite[Lemma 2.1]{LRLTT}. However, our initial data assumption is weaker, see Remark \ref{r:data}, so we include a proof for completeness. 
\begin{lemma}[Weak Observability Inequality]\label{weakobserve}
Suppose $Q$ satisfies Assumption \ref{ftGCC} and let $\chi_Q$ be the indicator function on $Q$. If $\p\Omega \neq \emptyset,$ assume moreover that no generalized bicharacteristic has contact of infinite order with $\p\Omega \times (0,T)$, that is $\Gc^{\infty}=\emptyset$.  Then there exists $C>0$, such that for all $(\psi_0, \psi_1) \in H \times L^2(\Omega)$ and $\psi$ solving \eqref{WEb} with initial data $(\psi_0, \psi_1)$, then
\begin{equation}
	\nm{\psi_0}_H^2 + \ltwo{\psi_1}^2 \leq C \left( \int_0^T \int_{\Omega} \chi_Q |\p_t \psi|^2 dx dt+ \ltwo{\psi_0}^2 + \nm{\psi_1}_{H'}^2 \right).
\end{equation}
\end{lemma}
\begin{proof}
Assume otherwise, so there exists a sequence of initial data $(\psi_{0,n}, \psi_{1,n}) \in H \times L^2(\Omega)$, with corresponding solution $\psi_n$, such that 
\begin{align}
&\nm{\psi_{0,n}}_{H}^2 + \ltwo{\psi_{1,n}}^2 = \ltwo{\nabla \psi_{0,n}} + \ltwo{\psi_{0,n}} + \ltwo{\psi_{1,n}}=1 \label{defectcontra}\\
&\limn \ltwo{\psi_{0,n}}^2 + \nm{\psi_{1,n}}_{H'}^2 = 0 \label{defectlone}\\
&\limn \int_0^T \int_{\Omega} \chi_Q |\p_t \psi_n|^2 dx dt=0. \label{defectdamping}
\end{align}
Note $(\psi_{0,n}, \psi_{1,n})$ is bounded in $H \times L^2(\Omega)$, so it contains a weakly convergent subsequence. By \eqref{defectlone} and the compact embedding of $H \times L^2(\Omega)$ into $L^2(\Omega) \times H'$, the weak limit can only be $(0,0)$. By Lemma \ref{weakconverge}, $\psi_n$ weakly converges to $0$ in $H^1(\Omega \times (0,T))$.

Now by Appendix \ref{defect}, up to replacement of $\psi_n$ by a subsequence, there exists a microlocal defect measure $\mu$ on $S^*\hat{\Sigma},$ such that for every $R \in \Psi_b^0(\Omega \times (0,T))$
\begin{equation}\label{eq:defectmeasure}
	\limn \<R\psi_n, \psi_n\>_{H^1(\Omega \times (0,T))} = \int_{S^* \hat{\Sigma}} \kappa(R) d \mu,
\end{equation}
where $\kappa(R)$ is the compressed principal symbol of $R$. See Appendix \ref{defect} for details on $\Psi_b^0$ and $\kappa$. By \eqref{defectdamping}, $\mu$ vanishes on $j(T^* Q) \cap S^* \hat{\Sigma}$. Note since $\Gc^{\infty} = \emptyset$, $\mu$ is invariant under the compressed generalized bicharacteristic flow by \cite[Lemma 2.1]{LRLTT} \cite[Section 3]{BurqLebeau2001},\cite[Section 2.2]{Lebeau1996}. The definition of the flow is given in Appendix \ref{nullbichar}. Now by Assumption \ref{ftGCC}, every unit-speed generalized bicharacteristic intersects $j(T^*Q)$. Since $\mu$ is invariant along generalized bicharacteristics, and $\mu=0$ on $j(T^* Q) \cap S^* \hat{\Sigma}$, then $\mu$ vanishes identically on $S^*\hat{\Sigma}$. Therefore, choosing $R=1$ in \eqref{eq:defectmeasure}, $u_n$ converges strongly to $0$ in $H^1(\Omega \times (0,T))$. Then note 
\begin{equation}
	0 = \limn \int_0^T \ltwo{\nabla \psi_n(\cdot, t)}^2 + \ltwo{\psi_n(\cdot, t)}^2 + \ltwo{\p_t \psi_n(\cdot, t)}^2  dt \geq \limn 2 \int_0^T E(\psi_n, t) dt.
\end{equation}
Since $\psi_n$ is a solution of the wave equation, $E(\psi_n,t)=E(\psi_n,0)$ for all $t,$ and so 
\begin{equation}
	\limn 2 E(\psi_n, 0) =\limn \left( \ltwo{\nabla \psi_{0,n}}^2 + \ltwo{\psi_{1,n}}^2 \right) = 0.
\end{equation}
This along with \eqref{defectlone} contradicts \eqref{defectcontra}. Thus the desired conclusion must hold.
\end{proof}

Following the typical approach to proving wave equation observability inequalities, we now seek to remove the $\ltwo{\psi_0}, \nm{\psi_1}_{H'}$ error terms from the right hand side of Lemma \ref{weakobserve}. In order to do so, we will first show that the only solutions of the wave equation \eqref{WEb} which are ``invisible" to the observation set $Q$ are the constant solutions. 

To make this precise we define the set of invisible solutions 
\begin{equation}
	N_T=\{v \in H^1(\Omega \times (0,T)); \Box v =0, (v,\p_t v)|_{t=0}= (v_0, v_1) \in H \times L^2(\Omega) \text{ and } \chi_Q \p_t v=0\}.
\end{equation}
We equip it with the norm
\begin{equation}
	\nm{v}_{N_T}^2 = \nm{v_0}_{H}^2 + \ltwo{v_1}^2,
\end{equation}
and note that if $u=v$ in $N_T$, then $u$ and $v$ solve the wave equation with the same initial data, and so $u=v$ almost everywhere on $\Omega \times [0,T]$. 

We now show that the invisible solutions consist only of constants. Our proof follows the same outline as \cite[Lemma 2.3]{LRLTT} but must change to handle our more general initial conditions. We will show that $N_T$ is finite dimensional and $\p_t$ maps $N_T/\{c\}$ to itself. So if $N_T\neq \{c\}$, then $\p_t$ has a non-trivial eigenvalue. Then proceeding by contradiction, we study the non-trivial eigenfunction of $\p_t$, and using elliptic unique continuation show it must be $0$, giving the desired contradiction.
\begin{lemma}\label{invisiblelemma}
$N_T=\{c\}$, the constant functions. 
\end{lemma}
\begin{proof}
To begin we follow the approach of \cite[Lemma 2.3]{LRLTT}. Note for all $v \in N_T$, by the weak observability inequality Lemma \ref{weakobserve}
\begin{equation}
	\nm{v}_{N_T} =\nm{v_0}_{H}^2 + \ltwo{v_1}^2 \leq C \left( \ltwo{v_0}^2 + \nm{v_1}_{H'}^2\right).
\end{equation}
By Rellich-Kondrachov, $H(\Omega) \times L^2(\Omega)$ is compactly embedded in $L^2(\Omega) \times H'(\Omega)$. This along with the above inequality implies that the unit ball in $N_T$ is compact, so $N_T$ is finite dimensional. 

Now, if $v \in N_T$, we claim that $\p_t v \in N_T$ as well. To see this first note that since $\chi_Q \p_t v =0$ and every generalized geodesic passes through $Q$, by propagation of singularities, \cite{MelroseSjostrand1978,ms1982}, $\p_t v$ is smooth in $\Omega \times (0,T)$. 
Therefore, $\p_t v \in H^1(\Omega \times (0,T) )$, and  $\p_t v|_{t=0} \in H(\Omega), \p_t^2 v|_{t=0} \in L^2(\Omega)$. 
Since $\Box v=0$, it is immediate that $\Box \p_t v=0$. 
Finally, since $\chi_Q \p_t v=0$, $\p_t v$ is constant on the open set $Q$ and so $\chi_Q \p_t^2 v=0$ as well. 
So indeed, if $v \in N_T$, then $\p_t v \in N_T$.

At this point to address our different initial conditions the details of our approach differ from that of \cite{LRLTT}. We now quotient out by the constant functions 
\begin{equation}
	N_T/\{c\} = \{ [v]; v \in N_T, v_1 \sim v_2 \text{ if } v_1+c=v_2 \text{ for some constant } c\}.
\end{equation}
Since $N_T$ is finite dimensional, and $\{c\}$ is a subspace, then $N_T/\{c\}$ is also finite dimensional. Note also because $\p_t$ maps $N_T$ to $N_T$, then $\p_t$ maps $N_T/\{c\}$ to $N_T/\{c\}$. 

Now, assume $N_T \neq \{c\}$ and a contradiction will be produced. Since $N_T/\{c\}$ is finite dimensional and nontrivial, $\p_t: N_T/\{c\} \ra N_T /\{c\}$ has at least one eigenvalue $\lambda$ associated to a nontrivial eigenfunction $v$. 

We first claim that $\lambda \neq 0$. To see this consider $v \in N_T /\{c\}$ with $\p_t v=0$, and we will show $v=0$. Since $\p_t v=0$ in $N_T/\{c\}$, then $\p_t v=c$ in $N_T$. But since $v \in N_T$, then 
\begin{equation}
	0=\chi_Q \p_t v = \chi_Q c.
\end{equation}
Since $Q$ is open and non-empty this means $0=c$, and so $0 =\p_t v$ in $N_T$. Then $v(x,t)=v(x)$, and $\Box v=0$ implies $-\Delta v=0$. The only harmonic functions on $\Omega$ compact are constants, so $v=c$ in $N_T$, and $v=0$ in $N_T/\{c\}$. Therefore there are no nontrivial eigenfunctions with $\lambda=0$, and so we must have $\lambda \neq 0$. 

Now consider this non-trivial eigenfunction of $\p_t$, $v$, such that $\p_t v= \lambda v$ in $N_T/\{c\}$ with $\lambda \neq 0$. Then $\p_t v = \lambda v + c$ in $N_T$ for some constant $c$. Thus for some $q(x)$ 
$$
v(x,t) = e^{\lambda t}q(x) - \frac{c}{\lambda}.
$$ 
Since $(\p_t^2 -\Delta)v(x,t) =0,$ then $(\lambda^2-\Delta) q(x)=0$.

Now take any $t \in (0,T),$ such that $\omega(t) = \{x; (x,t) \in Q\}$ contains a nonempty open set. Since $\chi_Q \p_t v =0,$ then $\chi_Q q(x)=0,$ so $q(x)=0$ on the open set $\omega(t)$. Then by elliptic unique continuation, $q \equiv 0$ on $\Omega,$ and so $v\equiv 0$ in $N_T/\{c\}$, which is a contradiction. Thus $N_T =\{c\}$ as desired. 
\end{proof}
Now using this characterization of invisible solutions, we can remove the error terms $\ltwo{\psi_0}$ and $\nm{\psi_1}_{H'}$ from the right hand side of Lemma \ref{weakobserve} to obtain an observability inequality for the wave equation.
\begin{lemma}\label{observeprop}
Suppose $Q$ satisfies Assumption \ref{ftGCC} and let $\chi_Q$ be the indicator function on $Q$. If $\p\Omega \neq \emptyset,$ assume moreover that no generalized bicharacteristic has contact of infinite order with $\p\Omega \times (0,T)$, that is $\Gc^{\infty}=\emptyset$. Then there exists $C_2>0,$ such that for all  $\psi$ solving \eqref{WEb}
\begin{equation}\label{eq:observe}
\frac{1}{2} \left( \ltwo{\nabla \psi_0}^2 + \ltwo{\psi_1}^2\right) = E(\psi,0) \leq C_2 \int_0^T \int_{\Omega} \chi_Q |\p_t \psi|^2  dx dt. %C\int_0^T \ltwo{G \p_t u}^2 dt.
\end{equation}
\end{lemma}
Note that this result is still weaker than Proposition \ref{uniobserveprop} because we are only working on the time interval $(0,T)$. At this point, we have not yet obtained a constant that is uniform for all time intervals $(t_0, t_0+T)$.

Note, when the boundary condition is Dirichlet, this result follows immediately from \cite[Theorem 1.8]{LRLTT}. So in the proof we assume that the boundary condition is Neumann or the boundary is empty. In particular we can write $H=H^1(\Omega)$ and $H'=H^{-1}(\Omega)$.

The proof will proceed by contradiction. Given a contradictory sequence, we can extract a weakly convergent subsequence. The weak limit will be an invisible solution and combining this with Rellich-Kondrachov and Lemma \ref{weakobserve} will provide the desired contradiction. 

\begin{proof}[Proof of Lemma \ref{observeprop}]
To begin, the inequality will be shown for solutions with initial position data having average value 0 and then the result will be extended to general initial data. So to begin assume $\int_{\Omega} \psi_0 dx=0$, and \eqref{eq:observe} will be shown. Notice that $\int_{\Omega} \psi_1 dx$ need not be $0$, and it is this case that cannot be handled directly from the existing result.

We will proceed by contradiction, so assume there exists a sequence $(\psi_{0,n}, \psi_{1,n}) \in H^1(\Omega) \times L^2(\Omega)$ with $\int_\Omega \psi_{0,n} dx =0$, such that 
\begin{equation}\label{defecterrorcontra}
\ltwo{\nabla \psi_{0,n}}^2+\ltwo{\psi_{1,n}}^2 =1, \qquad \limn \int_0^T \int_{\Omega} \chi_Q |\p_t \psi_n|^2 dx dt = 0,
\end{equation}
where $\psi_n$ is the solution of \eqref{WEb} with initial data $(\psi_n, \p_t \psi_n)|_{t=0}=(\psi_{0,n}, \psi_{1,n})$. 

The sequence $(\psi_{0,n},\psi_{1,n})$ is bounded in $H^1(\Omega) \times L^2(\Omega)$, so there exists a weakly convergent subsequence with limit $(\psi_0,\psi_1) \in H^1(\Omega) \times L^2(\Omega)$. Let $\psi$ solve \eqref{WEb}, with initial data $(\psi_0,\psi_1)$. Then by Lemma \ref{weakconverge}, $\p_t \psi_n \rhu \p_t \psi$ weakly in $L^2(\Omega \times (0,T))$. Therefore $\chi_Q \p_t \psi_n \rhu \chi_Q \p_t \psi$ weakly in $L^2(\Omega \times (0,T))$ and so
\begin{equation}
	\int_0^T \int_{\Omega} \chi_Q |\p_t \psi|^2 dx dt \leq \liminf_{n \ra \infty} \int_0^T \int_{\Omega} \chi_Q |\p_t \psi_n|^2 dx dt =0.
\end{equation}

Thus $\psi \in N_T$ and by Lemma \ref{invisiblelemma}, $\psi=c$, for some constant $c$. So $(\psi_0, \psi_1)=(c,0)$ and $(\psi_{0,n}, \psi_{1,n})$ converges to $(c,0)$ weakly in $H^1(\Omega) \times L^2(\Omega)$.  By Rellich-Kondrachov, this convergence is strong in $L^2(\Omega) \times H^{-1}(\Omega)$. But since $\int_\Omega \psi_{0,n} dx=0$ for all $n$, then $c=0$, so $(\psi_{0,n}, \psi_{1,n})$ strongly converges to $(0,0)$ in $L^2(\Omega) \times H^{-1}(\Omega)$. Put another way 
\begin{equation}\label{eq:psiweak0}
	\limn \lp{\psi_{0,n}}{2} + \nm{\psi_{1,n}}_{H^{-1}} = 0. 
\end{equation}

On the other hand by Lemma \ref{weakobserve} 
\begin{equation}
	\ltwo{\nabla \psi_{0,n}}^2+\ltwo{\psi_{1,n}}^2 \leq \hp{\psi_{0,n}}{1}^2 + \ltwo{\psi_{1,n}}^2 
	\leq C \left( \ltwooo{\chi_Q \p_t \psi_n}^2 +  \lp{\psi_{0,n}}{2} + \nm{\psi_{1,n}}_{H^{-1}} \right).
\end{equation}
By \eqref{defecterrorcontra} the left hand side equals 1 for all $n$, while by \eqref{defecterrorcontra} and \eqref{eq:psiweak0} the right hand side goes to $0$ as $n \ra \infty$, which is a contradiction. So \eqref{eq:observe} holds when $\int_\Omega \psi_0 dx=0$. 

When 
\begin{equation}
	\frac{1}{\text{vol}(\Omega)}\int_\Omega \psi_0 dx =\overline{\psi_0} \neq 0,
\end{equation}
note that if $\psi(x,t)$ solves $\Box \psi=0$ with initial data $(\psi_0,\psi_1)$, then  $\ti{\psi}(x,t)=\psi(x,t)-\overline{\psi_0}$ solves $\Box \ti{\psi}=0$ with initial data $(\psi_0-\overline{\psi_0}, \psi_1)$. The observability inequality can be applied to $\ti{\psi}$ and the constant $\overline{\psi_0}$ drops out due to the derivatives. This proves \eqref{eq:observe} for general initial data.
\end{proof} 

\subsection{Proof of Uniform Wave Observability, Proposition \ref{uniobserveprop}}\label{propsection}
%\begin{proof}
%Let $\theta=u-\psi$ so $(\p_t^2 -\Delta) \theta = -W \p_t u$ and $(\theta,\p_t \theta)|_{t=0} =(0,0)$. Now recall $E(\theta, t) = \frac{1}{2} \int_\Omega |\nabla \theta|^2 + |\p_t \theta|^2 dx$ and so, by a density argument
%\begin{equation}\label{thetaint}
%\p_t E(\theta,t) = \int_\Omega \p_t \theta (\p_t^2 -\Delta) \theta dx = \int_\Omega - W \p_t u \p_t \theta dx. 
%\end{equation}
%Now, using that $E(\theta,0)=0$ 
%$$
%\int_0^T E(\theta, s) ds = \int_0^T \int_0^s \p_t E(\theta, t) dt ds.
%$$
%Apply Fubini's Theorem to the right hand side 
%\begin{align}
%\int_0^T E(\theta,s) ds&=\int_0^T \int_t^T \p_t E(\theta,t) ds dt \\
%&= \int_0^T (T-t) \p_t E(\theta, t) dt \\
%&=\int_0^T (T-t) \int_\Omega - W \p_t u \p_t \theta dx dt & \text{ by } \eqref{thetaint}\\
%&\leq T^2 \int_0^T \int_\Omega |W \p_t u|^2 dx dt + \frac{1}{4} \int_0^T \int_\Omega |\p_t \theta|^2 dx dt. 
%\end{align}
%Where the inequality in the last line follows by applying Young's inequality for products. Recalling $E(\theta, t) \geq \frac{1}{2} \int_\Omega |\p_t \theta|^2 dx$ the second term on the right hand side can be absorbed, producing 
%$$
%\int_0^T \int_\Omega |\p_t \theta|^2 dx dt \leq 2 T^2 \int_0^T \int_\Omega |W \p_t u|^2 dx dt. 
%$$
%Now since $\psi=u-\theta$, and applying the previous line
%\begin{align}
%\int_0^T \int_\Omega |W \p_t \psi|^2 dx dt &\leq 2 \int_0^T \int_\Omega |W \p_t u|^2 dx dt + 2 \int_0^T \int_\Omega |W \p_t \theta|^2 dx dt \\
%&\leq \left(2+4T^2 \lp{W}{\infty}^2 \right) \int_0^T \int_\Omega |W \p_t u|^2 dx dt, 
%\end{align}
%which is exactly the desired inequality. 
%\end{proof}
With Lemma \ref{observeprop}, Proposition \ref{uniobserveprop} can be proved via a contradiction argument. Note that in this subsection the boundary conditions can be Dirichlet or Neumann, or the boundary can be empty.
\begin{proof}[Proof of Proposition \ref{uniobserveprop}]
Assume the desired conclusion does not hold, so for some $T \geq T_0$, there exist sequences $t_j \in [0,\infty)$ and $\psi_j \in L^2((0, \infty); H)$ solving \eqref{WEb}, with $\p_t \psi_j \in L^2((0, \infty); L^2(\Omega))$ such that 
\begin{equation}
	E(\psi_j, t_j) =1, \qquad \limj \int_{t_j}^{t_j+T}\int_{\Omega} W |\p_t \psi_j|^2 dx dt= 0.
\end{equation}
%Note 
%$$
%\limj \ltwooj{W \p_t u_j} \leq  \limj \lp{W}{\infty}^{1/2} \ltwooj{W^{1/2} \p_t u_j} =0.
%$$
Then let $v_j(x,t) = \psi_j(x,t+t_j)$ and $W_j(x,t) = W(x,t+t_j),$ so $ (\p_t^2 +A) v_j = 0$
and 
\begin{equation}\label{contra}
E(v_j, 0) =1, \qquad \limj \int_0^T \int_{\Omega} W_j  |\p_t v_j|^2 dx dt = 0.
\end{equation}
To relate the quantities in \eqref{contra} and obtain a contradiction, we would like to apply the observability inequality of Lemma \ref{observeprop}. In order to do so in a uniform fashion, we must replace $W_j$ by a fixed function satisfying Assumption \ref{ftGCC}, which is still close to the $W_j$. 
%Now note that for any $k$, $W_j$ is a bounded sequence in $H^k(\Omega \times (0,T))$ and so there exists a weakly convergent subsequence with limit $\Winf$. By Rellich-Kondrachov, this convergence is strong in $H^{k-1}(\Omega \times (0,T))$, after potentially replacing $W_j$ by a subsequence. In particular $W_j \ra \Winf$ in $L^{\infty}(\Omega \times (0,T))$.

To do so, first note that $\{W_j\}$ forms a pointwise bounded family, since $W \in L^{\infty}(\Omegab \times [0,\infty))$, and $\{W_j\}$ is an equicontinuous family in $C(\Omegab \times [0,T])$, by the uniform continuity of $W$ on $\Omegab \times [0,\infty)$. Therefore by Arzel\`a-Ascoli \cite[Theorem 7.25] {BabyRudin} there exists $W_{\infty} \in C(\Omegab \times [0,T])$ such that, after potentially replacing $W_j$ by a subsequence, $W_j \ra W_{\infty}$ in $L^{\infty}(\Omegab \times [0,T])$. 

Now recall $\Cm$ from Assumption \ref{TGCCp}, and we claim that $\left\{\Winf>\frac{\Cm}{2}\right\}$ satisfies Assumption \ref{ftGCC}. To see this choose $J$ large enough so that $\|\Winf-W_j\|_{L^{\infty}(\Omega \times [0,T])}< \frac{\Cm}{2}$ for $j \geq J$. For any generalized geodesic $\gamma(t)$, since $T \geq T_0$, by Assumption \ref{TGCCp}
\begin{align}
\frac{1}{T} \int_0^T \Winf(\gamma(t),t)dt &= \frac{1}{T} \int_0^T W_j(\gamma(t),t) dt + \frac{1}{T} \int_0^T \Winf(\gamma(t),t)-W_j(\gamma(t),t)dt\\
&\geq \frac{1}{T} \int_0^T W(\gamma(t),t+t_j) dt - \frac{1}{T} \int_0^T \|\Winf-W_j\|_{L^{\infty}(\Omega \times [0,T])}  dt \\
&\geq \Cm - \frac{\Cm}{2}=\frac{\Cm}{2}.
\end{align}
That is, for each generalized geodesic $\gamma$ the average of $\Winf$ along $\gamma$ is at least $\frac{\Cm}{2}$. Therefore, for each generalized geodesic $\gamma$, there exists $s_0 \in [0,T]$ such that  $\Winf(\gamma(s_0), s_0) \geq \frac{\Cm}{2}$. That is $\{W_{\infty}>\frac{\Cm}{2}\}$ satisfies Assumption \ref{ftGCC}. 

So now applying the observability inequality, Lemma \ref{observeprop} or \cite{LRLTT} Theorem 1.8, there exists $C_2>0$ such that 
\begin{equation}
	E(v_j, 0) \leq C_2 \int_0^T \int_{\Omega} W_{\infty} |\p_t v_j|^2 dx dt.
\end{equation}
Note that this $C_2$ is uniform in $j$ because the observation set $\{W_{\infty}>\frac{\Cm}{2}\}$ does not depend on $j$. 
To relate this inequality back to \eqref{contra}, the $W_{\infty}$ on the right hand side must be replaced by $W_j$. We can do so directly
\begin{align}\label{wjalign}
 E(v_j, 0) & \leq  C_2 \int_{0}^T \int_{\Omega} W_{\infty} |\p_t v_j|^2 dx dt \nonumber \\
&= C_2 \int_0^T \int_{\Omega} |W_j -W_j+W_{\infty}| |\p_t v_j |^2 dx dt \nonumber \\
& \leq  C_2 \int_0^T \int_{\Omega} W_j |\p_t v_j|^2dx dt  + C_2 \lp{W_j-W_{\infty}}{\infty}  \int_0^T \int_{\Omega} |\p_t v_j|^2 dx dt 
\end{align}
Since $v_j$ solves the wave equation, recall $E(v_j,0)=E(v_j,t)$ for all $t$. Thus 
\begin{equation}
	\int_0^T \int_{\Omega} |\p_t v_j|^2 \leq 2\int_0^T E(v_j,t) dt = 2T E(v_j, 0).
\end{equation} 
Now choose $J$ large enough, such that $\lp{W_j - W_{\infty}}{\infty} < \frac{1}{4T C_2}$ for $j  \geq J$. Then
\begin{equation}
	C_2 \lp{W_j - W_{\infty}}{\infty}\int_0^T \int_{\Omega} |\p_t v_j|^2 dx dt \leq \frac{1}{2}E(v_j,0).
\end{equation}
This term can be absorbed back into the left hand side of \eqref{wjalign} to give
\begin{equation}\label{eq:obscontra}
	\frac{1}{2} E(v_j,0) \leq C_2 \int_0^T \int_{\Omega} W_j |\p_t v_j|^2 dx dt.
\end{equation}
%Now, by Lemma \ref{observeconnection}
%\begin{align}
%E(v_j,0) &\leq 2C_2 \left(1+2T \nm{W_j}_{L_0^{\infty}}\right)^2 \ltwooo{W_j^{1/2} \p_t v_j}^2\\
%& \leq  2C_2 \left(1+2 T\nm{W}_{L^{\infty}(\Omega \times [0,\infty))} \right)^2 \ltwooo{W_j^{1/2} \p_t v_j}^2.
%\end{align}
By the first part of \eqref{contra} the left hand side equals $\frac{1}{2}$, while by the second part of \eqref{contra} the right hand side goes to $0$ as $j \ra \infty$. This is a contradiction, so the desired conclusion must hold. 
\end{proof}

\begin{remark}
Although the contradiction argument concludes with an inequality \eqref{eq:obscontra} exactly matching the form of Proposition \ref{uniobserveprop}, the proof cannot be easily rewritten to proceed directly. This is because the observability constant $C_2$ is only uniform by virtue of the contradiction argument. Proceeding directly from Lemma \ref{observeprop} does not work because the observability constant $C_2$ depends on the behavior of $W$ on $[t_0, t_0+T]$ and thus may change as $t_0$ changes.  
\end{remark}

\section{Upper bound on uniform stabilization rates, proof of Theorem \ref{bdecaythm}}\label{s:bdecaythm}
In this section we obtain uniform stabilization rates for damping that potentially tends to $0$ as $t \ra \infty$, Theorem \ref{bdecaythm}. We do so by first connecting observability estimates for the wave equation with observability estimates for the damped wave equation. These damped wave observability estimates are then converted to a uniform stabilization rate, as observability estimates give control over the time derivative of the energy. We finally show that if the damping satisfies the hypotheses of Theorem \ref{bdecaythm}, then solutions of \eqref{TDWEb} satisfy a damped wave observability inequality. From this we will conclude the desired uniform stabilization rate. See \cite{Haraux1989} for an analogous argument when the damping does not depend on time. 

We begin with an equivalence between observability for the damped wave equation and observability for the standard wave equation with the same initial data, when the observability operator is the damping. 
The exact statement used here is \cite[Lemma 3.3]{PaunonenSeifert2019}. 
\begin{lemma}\label{observeconnection}
Let $(u_0, u_1) \in H \times L^2(\Omega)$. Suppose $W \in L^{\infty}(\Omega \times [0,T])$, $u$ solves \eqref{TDWEb} and $\psi$ solves \eqref{WEb} with 
\begin{equation}
	(u, \p_t u)|_{t=0}= (\psi, \p_t \psi)|_{t=0}=(u_0, u_1).
\end{equation}
Then for any $t_0, T \geq 0$, letting $C_T=1+2T \lp{W}{\infty}$
\begin{equation}
	\int_{t_0}^{t_0+T} \int_{\Omega} W |\p_t u|^2 dx dt \leq \int_{t_0}^{t_0+T} \int_{\Omega} W |\p_t \psi|^2 dx dt \leq C_T^2 \int_{t_0}^{t_0+T} \int_{\Omega} W| \p_t u|^2 dx dt.
\end{equation}
\end{lemma}
%Together Proposition \ref{observeprop} and Lemma \ref{observeconnection} immediately give the following result.
%\begin{lemma}\label{propofsing} If $W \in C_u^{0}(\Omegab \times [0,\infty))$ satisfies Assumption \ref{TGCCp} and $T>T_0$, then there exists $C_1>0,$ such that for all $t_0 \in [0,\infty)$ and all solutions $u$ of \eqref{TDWEb}
%$$
%C_2 E(u,t_0)  \leq  \int_{t_0}^{t_0+T} \int_{\Omega} W |\p_t u|^2 dx dt. 
%$$
%In particular $C_2=\frac{C_1}{(1+2T \lp{W}{\infty})^2}$, where the $L^{\infty}$ norm of $W$ is taken over $\bar{\Omega} \times [0,\infty)$. 
%\end{lemma}
%It is perhaps unclear from the proof of Proposition \ref{propofsing} why Assumption \ref{TGCC} is natural. Because of this, a proof of a weaker version, that still has a constant uniform in $t_0$ and which naturally relies on Assumption \ref{TGCC}, is proved in Appendix \ref{propappendix}.
%First, I will introduce the notion of observability for wave equations. The wave equation is 
%\begin{equation}\label{WE}
%\begin{cases}
%(\p_t^2 -\Delta_g )\psi=0 \\
%(\psi,\psi_t)|_{t=0} =(\psi_0, \psi_1) \in H^1(M) \times L^2(M).
%\end{cases}
%\end{equation}
%We say \eqref{TDWE} (or \eqref{WE} resp.) is observable by $V: M \times [0,T] \ra [0, \infty)$ in time $T$, if there exists $C>0$, such that 
%$$
%E(u,0) \leq C \int_0^T \int_M V(x,t) |\p_t u|^2 dx dt,
%$$
%for all solutions $u$ of either \eqref{TDWE} (or \eqref{WE} resp.). Such an inequality is called an observability inequality and $V$ is called the observability operator. 
Now we relate damped wave observability estimates with uniform stabilization rates. Note that we allow the observability constant to depend on the starting time of the observation time interval.
\begin{lemma}\label{observetodecay}
Suppose $u$ solves \eqref{TDWEb}, and let $b: \Nb \ra [0,\infty)$. Assume there exists $T_0>0,$ such that for any $j \in \Nb_0$
$$
b(jT_0) E(u,jT_0) \leq \int_{jT_0}^{(j+1)T_0} \int_{\Omega} W |\p_t u|^2 dx dt. 
$$
Then, defining $B(k)=\sum_{j=0}^{k-1} b(jT_0)$, we have 
$$
E(u,kT_0) \leq E(u,0) \exp(-B(k)).
$$
\end{lemma}
\begin{proof}
To begin note 
\begin{equation}
	\p_t E(u,t) = -\int_{\Omega} W |\p_t u(x,t)|^2 dx.
\end{equation}  
Integrating this from $jT_0$ to $(j+1)T_0$, then applying the assumed observability estimate 
\begin{align}
E(u,(j+1)T_0) - E(u,jT_0) &= - \int_{jT_0}^{(j+1)T_0} \int_{\Omega} W |\p_t u|^2 dx dt \\
&\leq -b(jT_0) E(u,jT_0).
\end{align}
Rearranging we see 
\begin{equation}
	E(u,(j+1)T_0) \leq (1- b(jT_0)) E(u,jT_0).
\end{equation} 
So a fraction of the energy, specified by $1-b(jT_0)$, is lost from time $jT_0$ to $(j+1)T_0$.

Applying this iteratively from $j=0$ to $j=k-1$ we obtain
\begin{equation}\label{eq:EukTfinal}
	E(u,kT) \leq  E(u,0) \prod_{j=0}^{k-1} (1- b(jT_0)).
\end{equation}
Now noting that $\ln(x) \leq x-1$, we have 
\begin{equation}
	\ln \left( \prod_{j=0}^{k-1} (1-b(jT_0)) \right)=\sum_{j=0}^{k-1} \ln(1- b(jT_0)) \leq \sum_{j=0}^{k-1} - b(jT_0) = - B(k).
\end{equation}
Exponentiating both sides of this equation and applying \eqref{eq:EukTfinal} gives the desired conclusion. 
\end{proof}
\begin{remark}
	Some structure is required for time-dependent damping to obtain such an observability estimate. In the autonomous case, this $b$ is related to $\Sigma(t)$ or $L(t)$, but explicit formulas require taking $t \ra \infty$ \cite{HumbertPrivatTrelat2019} or give rough bounds like \cite[Theorem 1.2]{LaurentLeautaud2016} 
	\begin{equation}
		C \exp(-c(\Sigma((j+1)T_0)-\Sigma(jT_0))^{-1}) \leq b(jT_0).
	\end{equation}
Compare summing these $b(jT_0)$ to get $B(k)$ with our eventual estimate of $B(k)$ \eqref{eq:sigmadecrease}.
There are not analogous results for general time-dependent damping. In this case, we avoid this issue by fixing the observation function $\ti{W}$ to obtain a uniform observability constant, and then separately track how the damping tending to $0$ changes the observability inequality. See also Section \ref{examplesection} for additional examples where uniform stabilization rates are obtained via observability estimate arguments.
\end{remark}

We now prove such an observability estimate for damping satisfying the hypotheses of Theorem \ref{bdecaythm}. Note that the size of the observability constant depends on the decreasing function $b(t)$.
\begin{lemma}\label{decreasingdwobserve}
Suppose $u$ solves \eqref{TDWEb}, with $W(x,t)=f(x,t) \ti{W}(x,t)$, $b(t) \simeq f(x,t)$, and $T_0 > 0$ all as in Theorem \ref{bdecaythm}. Then there exists $C >0$, such that for all $t_0 \geq 0$
\begin{equation}
C b(t_0+T_0) E(u,t_0) \leq \int_{t_0}^{t_0+T_0} \int_{\Omega} W|\p_t u|^2 dx dt.
\end{equation}
\end{lemma}
\begin{proof}
Let $\psi$ solve \eqref{WEb} with $(\psi, \p_t \psi)|_{t=t_0} = (u, \p_t u)|_{t=t_0}$. Since  $\ti{W}$ satisfies Assumption \ref{TGCCp}, by Proposition \ref{uniobserveprop} there exists $C_1>0$ such that for all $t_0 \geq 0$
\begin{equation}
E(u,t_0) = E(\psi,t_0) \leq C_1 \int_{t_0}^{t_0+T_0} \int_{\Omega} \ti{W}(x,t) |\p_t \psi|^2 dx dt.
\end{equation}
Then since $\ti{W}(x,t)=\frac{W(x,t)}{f(x,t)}$, $f(x,t) \geq c_m b(t),$ and $b$ is decreasing
\begin{equation}
E(u,t_0) \leq C_1 \int_{t_0}^{t_0+T_0} \int_{\Omega} \frac{W(x,t)}{f(x,t)} |\p_t \psi|^2 dx dt \leq \frac{C_1}{c_m b(t_0+T_0)} \int_{t_0}^{t_0+T_0} \int_{\Omega} W |\p_t \psi|^2 dx dt.
\end{equation}
Applying Lemma \ref{observeconnection} 
\begin{equation}
E(u,t_0) \leq \frac{C_1}{c_m b(t_0+T_0)} (1+2 T_0 \lp{W}{\infty})^2 \int_{t_0}^{t_0+T_0} \int_{\Omega} W |\p_t u|^2 dx dt. 
\end{equation}
Multiplying both sides by $\frac{c_m b(t_0+T_0)}{C_1 (1+2T_0 \lp{W}{\infty})^2}$ gives the desired conclusion. Note that we require $W \in L^{\infty}(M \times [0,\infty))$ to complete this step, which is why we only consider $b(t)$ decreasing, and do not consider $b(t)$ increasing without bound.
\end{proof}

Now we can complete the proof of Theorem \ref{bdecaythm}. We apply Lemma \ref{decreasingdwobserve}  to obtain damped wave equation observability inequalities, then apply Lemma \ref{observetodecay} to convert the observability inequalities into an inequality involving $E(u,t)$ and $B(k)$, and finally estimate $B(k)$ to obtain a uniform stabilization rate for solutions of \eqref{TDWE}.
\begin{proof}[Proof of Theorem \ref{bdecaythm}]
By Lemma \ref{decreasingdwobserve}, there exists $C>0$ such that for all $j \in \mathbb{N}$
\begin{equation}
C b((j+1)T_0) E(u,jT_0) \leq \int_{jT_0}^{(j+1)T_0} \int_M W |\p_t u|^2 dx dt. 
\end{equation}
Then by Lemma \ref{observetodecay}, letting $B(k) = \sum_{j=0}^{k-1} C b((j+1)T_0)$
\begin{equation}\label{eq:Bkenergydecrease}
E(u,kT_0) \leq E(u,0) \exp(-B(k)).
\end{equation}
Now since $b$ is decreasing, we have $b(a) \geq \frac{1}{T_0} \int_a^{a+T_0}b(x) dx$, and so we can directly estimate
\begin{align}
B(k) &\geq \sum_{j=0}^{k-1} C \int_{(j+1)T_0}^{(j+2)T_0} b(t) dt = \frac{1}{T_0} \int_{T_0}^{(k+1)T_0} C b(t) dt 
\\
&= \frac{1}{T_0} \int_0^{(k+1)T_0} Cb(t) dt - \frac{1}{T_0} \int_0^{T_0} Cb(t) dt  \\
&\geq \frac{1}{T_0} \int_0^{(k+1)T_0} C b(t) dt - C b(0).
\end{align}
Since $\lp{\ti{W}}{\infty} < \infty$ and $f(x,t) \leq C_M b(t)$, for any unit-speed generalized geodesic $\gamma$ we have 
\begin{equation}
	b(t) \geq \frac{1}{\lp{\ti{W}}{\infty} C_M} \ti{W}(\gamma(t), t) f(\gamma(t),t).
\end{equation}
Integrating both sides and recalling the definition of $\Sigma(t)$
\begin{align}
B(k) +Cb(0) &\geq \frac{1}{T_0} \int_0^{(k+1)T_0} b(t) dt \\
&\geq \frac{1}{C_M T_0 \lp{\ti{W}}{\infty} } \int_0^{(k+1)T_0} \ti{W}(\gamma(t),t) f(\gamma(t),t) dt \geq \frac{\Sigma((k+1)T_0)}{C_M T_0 \lp{\ti{W}}{\infty}}. \label{eq:sigmadecrease}
\end{align}
Combining \eqref{eq:Bkenergydecrease} and \eqref{eq:sigmadecrease}, there exists $C,c>0$ such that 
\begin{equation}
E(u,kT_0) \leq C E(u,0) \exp(-c \Sigma((k+1)T_0).
\end{equation}
Since $E(u,t)$ is non-increasing and $\Sigma(t)$ is increasing, for $t \in [kT_0, (k+1)T_0]$ we have
\begin{equation}
E(u,t) \leq E(u,kT_0) \leq C E(u,0) \exp(-c\Sigma((k+1)T_0)) \leq C E(u,0) \exp(-c\Sigma(t)).
\end{equation}
Since every $t \geq 0$ is contained in such an interval for some $k$, this gives the desired energy decay.
\end{proof}

\section{Non-exponential uniform stabilization rates}\label{examplesection}
In this section, upper and lower bounds on uniform stabilization rates are proved for three examples, which do not satisfy the time dependent geometric control condition, Assumption \ref{TGCC}.
\begin{enumerate}
	\item Damping tending to $0$ at a polynomial rate as $t \ra \infty$ .
	\item Damping which turns on and off, with fixed length time intervals where it is positive and growing length intervals where it is 0.
	\item Damping which turns on and off, with shrinking length time intervals where it is positive and fixed length intervals where it is 0.
\end{enumerate}
To obtain upper bounds on uniform stabilization rates, we follow the same general approach used in Section \ref{s:bdecaythm}. To summarize: we obtain quantitative observability estimates for solutions of the damped wave equation via observability estimates for the wave equation, then convert these into uniform stabilization rates.  We obtain lower bounds by estimating $\Sigma(t)$ and applying Theorem \ref{lowerboundthm}. 

One notable feature of the uniform stabilization rates that we obtain is that most are not exponential. This runs counter to the usual behavior of solutions to the damped wave equation with autonomous damping. When the damping is autonomous, solutions form a semigroup and all uniform stabilization is automatically at an exponential rate. Because of this, uniform stabilization is equivalent to the geometric control condition for autonomous damping. When the damping is time-dependent, solutions no longer form a semigroup. Thus uniform stabilization could occur at non-exponential rates, and need not be equivalent to the time-dependent geometric control condition. The examples in this section exactly show that this possibility occurs.

\subsection{Damping tending to $0$ polynomially in $t$}
First, as a direct application of Theorems \ref{lowerboundthm} and \ref{bdecaythm}, we obtain uniform stabilization rates for damping tending to $0$ at a polynomial rate as $t \ra \infty$. Note that both Theorems \ref{lowerboundthm} and \ref{bdecaythm} apply to damping tending to $0$ at non-polynomial rates, and we choose to work with polynomials just to simplify the exposition of this example. 
\begin{example}
Suppose $\p \Omega=\emptyset$, $W(x,t)=\ti{W}(x,t) f(x,t)$ as in Theorem \ref{bdecaythm}, and $b(t)=(t+1)^{-\beta} \simeq f(x,t)$ with $\beta \geq 0$.
\begin{enumerate}
	\item If $0 \leq \beta<1$, there exists $C,c>0$ such that for all $u$ solving \eqref{TDWE}
	\begin{equation}
		E(u,t) < C E(u,0) \exp(-c t^{1-\beta}).
	\end{equation}
	Furthermore, if $c> \frac{2C_M}{(1-\beta)}\lp{\ti{W}}{\infty},$ this inequality cannot hold for all $u$ solving \eqref{TDWE}.
	\item If $\beta=1$, there exists $C,c>0$ such that for all $u$ solving \eqref{TDWE}
	\begin{equation}
		E(u,t) < C E(u,0) (t+1)^{-c}.
	\end{equation}
	Furthermore,  if $c>2C_M \lp{\ti{W}}{\infty},$ this inequality cannot hold for all $u$ solving \eqref{TDWE}.
	\item If $\beta>1$, there can not be a uniform stabilization rate. 
\end{enumerate}
\end{example}
The presence, or respectively lack, of a uniform stabilization rate for $\beta \leq 1$, respectively $\beta>1$, agrees with the distinction between ``effective" and ``non-effective" dissipation of \cite{Wirth2004, Wirth2006, Wirth2007, Matsumura1977} but our result allows for more general damping. For example $f$ need not be decreasing in $t$ and $W$ can be (not-identically) 0. 

Note that when $\beta>0$, the time dependent geometric control condition, Assumption \ref{TGCC}, is not satisfied yet uniform stabilization occurs. Taking $\beta=0$, Assumption \ref{TGCC} is satisfied, and we recover Theorem \ref{mainresult}.
\begin{proof}
The explicit rates follow from estimating $\Sigma(t)$ on intervals $[(j-1)T_0, jT_0]$ using that $f(x,t)\geq (t+1)^{-\beta}$
\begin{align}
\Sigma(kT_0) &= \inf_{\gamma} \sum_{j=1}^{k} \int_{(j-1)T_0}^{j T_0} \ti{W}(\gamma(t),t) f(\gamma(t), t) dt \\
&\geq \inf_{\gamma} \sum_{j=1}^{k} C_m (j T_0 +1)^{-\beta} \int_{(j-1) T_0}^{j T_0} \ti{W}(\gamma(t), t) dt.
\end{align}
Then since $\ti{W}$ satisfies Assumption \ref{TGCC}
\begin{align}
\Sigma(kT_0) &\geq C_m \Cm T_0 \sum_{j=1}^{k}  (j T_0 +1)^{-\beta}. 
\end{align}
Now using that $(tT_0+1)^{-\beta}$ is decreasing in $t$ and making a change of variables to evaluate the integral 
\begin{align}
\Sigma(kT_0)&\geq c \int_1^{k+1} (tT_0 +1)^{-\beta} dt \\
&\geq c \int_{T_0}^{(k+1) T_0} (t+1)^{-\beta} dt \\
&= \begin{cases}
\frac{c}{1-\beta} \left( ((k+1) T_0+1)^{1-\beta} -(T_0+1)^{1-\beta} \right) &\beta \neq 1 \\
c \left( \ln((k+1)T_0+1) - \ln(T_0+1) \right) & \beta=1.
\end{cases}
\end{align}
Plugging this into Theorem \ref{bdecaythm}, relabeling $c$ and $C$, and using that $E(u,t)$ is non-increasing, as in the conclusion of the proof of Theorem \ref{bdecaythm}, gives the desired uniform stabilization rates. 

Analogously, we can crudely bound $\ti{W} \leq \lp{\ti{W}}{\infty}$ and $f(x,t) \leq C_M b(t)$ to obtain 
\begin{align}
\Sigma(t) &= \inf_{\gamma} \int_0^t \ti{W}(\gamma(s), s) f(\gamma(s),s) ds \\
&\leq \lp{\ti{W}}{\infty} C_M \int_0^t (s+1)^{-\beta} ds = \lp{\ti{W}}{\infty} C_M \begin{cases} 
\frac{1}{1-\beta}\left((t+1)^{1-\beta}-1 \right)& \beta \neq 1\\
\ln(t+1) & \beta =1.
\end{cases}
\end{align}
Plugging this into Theorem \ref{lowerboundthm} gives the desired lower bound. 
\end{proof}

\subsection{Oscillating damping, $0$ for growing length time intervals}
We now consider damping which alternates in $t$ between $0$ and a nontrivial function. We assume the damping is non-trivial for time intervals of a fixed size $L_0$, and is 0 for growing time intervals. We allow the damping to depend on $x$ as well, so long as it satisfies the finite time TGCC on the time-intervals where it is non-trivial. 

To make this exact, fix $L_0$ and suppose $\ti{W}(x,t) \in C^0(M \times [0, L_0])$. Assume that for some $\e>0$, $\{\ti{W} > \e\}$ satisfies Assumption \ref{ftGCC} on $M \times [0,L_0]$. Let $\Cm=\inf_{\gamma} \frac{1}{L_0} \int_0^{L_0} \ti{W}(\gamma(t), t) dt$, and let $f:[0,\infty) \ra [C_1, \infty)$ be an increasing function. This $f$ will define the length of the time intervals where the damping is $0$.

At the $k$th step, we will have $W=\ti{W}$  for a time interval of length $L_0$, and then have $W=0$ for a time interval of length $f(k+1)$. In particular, for $k \in \Nb,$ define $T_k = \sum_{j=1}^k f(j)$ and 
\begin{align}\label{growingdampdef}
&W(x,t) = \begin{cases}
	\ti{W}(x,t) & 0 \leq t \leq L_0 \\
	0 & L_0 < t < L_0+f(1)\\
	\ti{W}(x,t- kL_0-T_k), & kL_0 + T_k  < t < (k+1)L_0 + T_k \\
	0 & (k+1)L_0 + T_k < t < (k+1)L_0 + T_k+f(k+1),
\end{cases}
\end{align}

\begin{figure}[h]
\center
\begin{tikzpicture}
\node at (-0.3, 2) {$W$};
\node at (8, -.3) {$t$};
\draw [->] (0,0) -- (0,2.2);
\draw [->] (0,0) -- (8.4,0);
\node at (0,-.3) {$0$};
\node at (.8,-.4) {$\overbrace{f(1)}$};
\node at (2,-.4) {$\overbrace{f(2)}$};
\node at (4,-.4) {$\overbrace{f(3)}$};
\node at (6.5,-.4) {$\overbrace{f(4)}$};
\node at (.25,2.5) {$\underbrace{L_0}$};
\node at (1.25, 2.5) {$\underbrace{L_0}$};
\node at (2.75, 2.5) {$\underbrace{L_0}$};
\node at (5.25, 2.5) {$\underbrace{L_0}$};
\node at (7.75, 2.5) {$\underbrace{L_0}$};
\draw [-] (0,2)--(.5,2);
\draw[-] (.5,2)--(.5,0);
\draw[-] (.5,0)--(1,0);
\draw[-] (1,2)--(1,0);
\draw[-] (1,2)--(1.5,2);
\draw[-] (1.5,2)--(1.5,0);
\draw[-] (1.5,0)--(2.5,0);
\draw[-](2.5,2)--(2.5,0);
\draw[-](2.5,2)--(3,2);
\draw[-] (3,2)--(3,0);
\draw[-] (3,0)--(5,0);
\draw[-](5,2)--(5,0);
\draw[-](5,2)--(5.5,2);
\draw[-] (5.5,2)--(5.5,0);
\draw[-](5.5,0)--(7.5,0);
\draw[-](7.5,0)--(7.5,2);
\draw[-](7.5,2)--(8,2);
\draw[-](8,2)--(8,0);
\node at (8.4,1) {$\cdots$};
\end{tikzpicture}
\caption{Time intervals where $W(x,t)$ satisfying \eqref{growingdampdef} equals $\ti{W}$ or 0.}
\end{figure}

Because $W$ is on for a fixed time interval $L_0$ and $\ti{W}$ satisfies Assumption \ref{ftGCC} on each of these time intervals, it is straightforward to obtain a damped wave observability inequality from Proposition \ref{uniobserveprop} and Lemma \ref{observeconnection}. Therefore solutions lose a fixed fraction of their energy for each full time interval of length $L_0$ that the damping is non-trivial.
However, Lemma  \ref{observetodecay} cannot be used to obtain a uniform stabilization rate, as the observability estimates do not have the appropriate periodic structure. Instead, given a time $t$ we estimate the number of time intervals of length $L_0$ where the damping has been non-trivial.
This is done by taking the inverse of an integral of $f$, which, as seen in the examples, provides the appropriate qualitative uniform stabilization rate. 
\begin{proposition}\label{growingprop}
Suppose $W$ is as in \eqref{growingdampdef}. Define $F(k)=\int_1^{k+1} f(z) dz$ and let $F^{-1}(t)$ be the inverse of $F$. Then, there exists $C, c>0$ such that for all $u$ solving \eqref{TDWE}
\begin{equation}
	E(u,t) < C E(u,0) \exp\left(-c F^{-1}\left( \frac{C_1 t}{L_0+C_1}\right) \right).
\end{equation}
Furthermore, define $B(k)=\int_0^k f(z) dz$ and let $B^{-1}(t)$ be the inverse of $B$. Then
\begin{equation}
	E(u,t) < C E(u,0) \exp\left(-c B^{-1}(t) \right),
\end{equation}
cannot hold for all $u$ solving \eqref{TDWE} if $c > 2 \Cm L_0$.

In particular, suppose $C_1>0, \alpha \geq 0$ and $r>1$ 
\begin{enumerate}
	\item If $f(j)=C_1 j^{\alpha}$, then there exists $C,c>0,$ such that for all $u$ solving \eqref{TDWE}, 
	\begin{equation}
		E(u,t) < CE(u,0) \exp(-c t^{\frac{1}{\alpha+1}}).
	\end{equation}
	Furthermore, this cannot hold for all $u$ solving \eqref{TDWE} if $c > 2\Cm L_0 \left( \frac{\alpha+1}{C_1} \right)^{\frac{1}{\alpha+1}}$.
	\item If $f(j) = C_1 r^j$, then there exists $C,c>0,$ such that for all $u$ solving \eqref{TDWE}, 
	\begin{equation}
		E(u,t) < C E(u,0) \frac{1}{(t+1)^c}
	\end{equation}
	Furthermore, this cannot hold for all $u$ solving \eqref{TDWE} if $c > \frac{2 \Cm L_0}{\ln(r)}$. 
	\item If $f(j)=C_1 e^{j+e^{j}}$, then there exists $C,c>0,$ such that for all $u$ solving \eqref{TDWE}, 
	\begin{equation}
		E(u,t) < C E(u,0) \frac{1}{\ln(2+t)^{c}}.
	\end{equation}
	Furthermore, this cannot hold for all $u$ solving \eqref{TDWE} if $c>2\Cm L_0$. 
\end{enumerate}
\end{proposition}
\begin{remark}
	\begin{enumerate}
		\item Note that when $f$ is strictly increasing, the damping is $0$ for a growing proportion of times, and so the time dependent geometric control condition cannot be satisfied. 
		\item The integrals defining $F(k)$ and $B(k)$ should be understood as sums like $\sum_{j=1}^{k+1} f(j)$, but with intermediate values filled in, so $k$ can be taken a real number. Furthermore, because we invert $F$ and $B$ we construct them to be strictly increasing. This is why we do not use, for instance, $\sum_{j=1}^{\lfloor k\rfloor +1}  f(j).$
\end{enumerate}
\end{remark}
\begin{proof}
First note that because $f \geq C_1>0$, $F$ and $B$ are strictly increasing, and so $F^{-1}$ and $B^{-1}$ exist.

1) Now given a time $t>0$, let $N(t)$ be the number of times the damping has been non-trivial for a time interval of length $L_0$ before time $t$. We will first bound $N(t)$ from below in terms of $F^{-1}$, then show that solutions lose a fixed fraction of their energy each time the damping is nontrivial for a full time-interval of length $L_0$. We conclude by combining these two statements to obtain a uniform stabilization rate.

Note that by its definition $N(t)$ is the largest $N$ such that 
\begin{equation}\label{eq:NTlargest}
	N L_0 + \sum_{j=1}^{N-1} f(j) \leq t.
\end{equation}
Now let $K=\lfloor F^{-1}( \frac{C_1 t}{L_0+C_1}) \rfloor$ and we claim that $K \leq N(t)$. To see this, first note that since $f$ is increasing, $f(a) \leq \int_a^{a+1} f(x) dx$. Therefore
\begin{equation}\label{eq:KFineq1}
	K L_0 + \sum_{j=1}^{K-1} f(j) \leq KL_0 + \int_1^K f(z) dz = K L_0 + F(K-1).
\end{equation}
Now, since $C_1 \leq f$, $C_1 K \leq F(K)$. Therefore $K L_0 \leq C_1^{-1} L_0 F(K)$. Note also, since $F$ is increasing $F(K-1)\leq F(K)$. Plugging these back into \eqref{eq:KFineq1} we obtain
\begin{equation}
	KL_0 + \sum_{j=1}^{K-1} f(j) \leq \left( \frac{L_0}{C_1} +1 \right) F(K).
\end{equation}
Now since $K \leq F^{-1}(\frac{C_1 t}{L_0+C_1})$ and $F$ is increasing, $F(K)\leq \frac{C_1 t}{L_0+C_1}$ and so 
\begin{equation}
	KL_0 + \sum_{j=1}^{K-1} f(j) \leq \left( \frac{L_0}{C_1}+1 \right) \frac{C_1 t}{L_0+C_1} = t.
\end{equation}
So $K$ also satisfies \eqref{eq:NTlargest}. Since $N(t)$ was the largest value for which that inequality holds we indeed have $K=\lfloor F^{-1}( \frac{C_1 t}{L_0+C_1}) \rfloor \leq N(t)$. Now since $K$ and $N(t)$ are both integers, we have 
\begin{equation}
	F^{-1}\left(\frac{C_1 t}{L_0+C_1}\right) \leq N(t)+1.
\end{equation}
Now we will show that every solution of \eqref{TDWE} loses a fixed fraction of its energy over each time interval of length $L_0$ where $W=\ti{W}$. So let $u$ be a solution of \eqref{TDWE} and let $\psi$ solve \eqref{WEb} with 
\begin{equation}
	(\psi, \p_t \psi)=(u,\p_t u) \text{ at } t= kL_0+T_k.
\end{equation} 
Since $W$ satisfies Assumption \ref{ftGCC} on the time interval $[kL_0+T_k, (k+1)L_0+T_k]$, by Proposition \ref{uniobserveprop}, there exists $C>0$, uniform in $k$, such that 
$$
E\left(u,kL_0 + T_k \right)= E\left(\psi, kL_0+T_k\right) \leq C \int_{kL_0+T_k}^{(k+1)L_0 + T_k} \int_{\Omega} W |\p_t \psi|^2 dx dt. 
$$
Then by Lemma \ref{observeconnection}, there exists $C>0$, uniform in $k,$ such that 
\begin{equation}\label{onoff1}
E\left(u,kL_0+T_k \right) \leq C \int_{kL_0+T_k}^{(k+1) L_0+ T_k} \int_{\Omega} W |\p_t u|^2 dx dt. 
\end{equation}
Now integrating $\p_t E(u,t) = -\int_M W|\p_t u(x,t)|^2 dx$ in $t$ from $kL_0+T_k$ to $(k+1)L_0 + T_k$ 
\begin{equation}\label{onoff2}
E\left(u,(k+1)L_0 + T_k \right) - E\left(u,kL_0 + T_k\right) = - \int_{kL_0+T_k}^{(k+1) L_0 +T_k} \int_{\Omega} W|\p_t u|^2 dx dt.
\end{equation}
Combining \eqref{onoff1} and \eqref{onoff2}, there exists $C^*>0$, independent of $k$, such that 
\begin{equation}\label{eq:energyL0}
E\left(u,(k+1)L_0 + T_k\right) < (1-C^*) E\left(u,kL_0+T_k\right).
\end{equation} 
So indeed from time $kL_0+T_k$ to $(k+1)L_0+T_k$, where $W=\ti{W}$, every solution loses a fixed fraction of its energy. Now we will combine our estimate of $N(t)$ with this fact to obtain a uniform stabilization rate. 

At time $t$, there have been $N(t)$ intervals of length $L_0$ where $W=\ti{W}$, and so 
\begin{equation}
E(u,t) < (1-C^*)^{N(t)} E(u,0).
\end{equation} 
Now since $N(t) \geq F^{-1}\left(\frac{C_1t}{L_0+C_1}\right)-1$, after manipulating the exponential directly, there exists $C,c>0,$ such that 
\begin{equation}
E(u,t) < C \exp\left(-c F^{-1}\left(\frac{C_1 t}{L_0+C_1 } \right) \right) E(u,0). 
\end{equation} 
2) To see the lower bound on the uniform stabilization rate we bound $N(t)$ from above using $B^{-1}(t)$, then estimate $\Sigma(t)$ in terms of $N(t)$, and finally apply Theorem \ref{lowerboundthm}.

We claim that $J=\lceil B^{-1}(t) \rceil +1 \geq N(t)$. To see this, note that for $a \in \Rb$, since $f$ is increasing $f(a) \geq \int_{a-1}^a f(x) dx$. Now note that $B(t)$ is increasing and $J-1 \geq B^{-1}(t)$, so $B(J-1) \geq t$. Combining these together with the fact that $JL_0 \geq 0$, we compute
\begin{equation}
J L_0 + \sum_{j=1}^{J-1} f(j) \geq J L_0 + \int_0^{J-1} f(z) dz \geq B(J-1) \geq t. 
\end{equation} 
Since $J$ does not satisfy \eqref{eq:NTlargest} we must have $J \geq N(t)$.
Therefore
\begin{align}
\Sigma(t) &= \inf_{\gamma} \int_0^t W(\gamma(s), s) ds \\
&\leq (N(t)+1) \inf_{\gamma} \int_0^{L_0} \ti{W}(\gamma(s),s) ds\\
&\leq (J+1) \inf_{\gamma} \int_0^{L_0} \ti{W}(\gamma(s), s) ds \\
&= \Cm L_0 (J+1) \leq \Cm L_0(B^{-1}(t)+3).
\end{align} 
Note the final inequality follows from the definition of $J$. Now applying Theorem \ref{lowerboundthm} gives the second conclusion, where we note that $3\Cm L_0$ can be absorbed into the constant $C$ in $E(u,t) \leq CE(u,0) \exp(-\beta \Sigma(t))$.

To see the examples, we directly compute $F, B$ and their inverses, then apply the first two parts of this proposition, using elementary exponential rules to consolidate constants. 

3. When $f(j) = C_1 j^{\alpha}$, then 
\begin{equation}
	F(k) = C_1 \left( \frac{(k+1)^{\alpha+1}-1}{\alpha+1}\right), \text{ so }F^{-1}(s) = \left( \frac{(\alpha+1)s}{C_1}+1\right)^{\frac{1}{\alpha+1}}-1. 
\end{equation}

%Thus, there exists $C,c>0,$ such that 
%$$
%E(u,t) < C \exp\left(-c \left( \frac{(\alpha+1)t}{2L_0+C_1} + 1\right)^{\frac{1}{\alpha+1}} \right) E(u,0).
%$$
%Simplifying and relabeling $C, c$, this becomes
%$$
%E(u,t) < C \exp(-c t^{\frac{1}{\alpha+1}} ) E(u,0).
%$$
On the other hand, 
\begin{equation}
	B(k) = \frac{C_1 k^{\alpha+1}}{\alpha+1}, \text{ so } B^{-1}(s) = \left( \frac{(\alpha+1)s}{C_1}\right)^{\frac{1}{\alpha+1}}.
\end{equation} %Therefore  
%$$
%E(u,t) < C E(u,0) \exp\left(-c t^{\frac{1}{\alpha+1}}\right),
%$$
%cannot hold with $c> 2\Cm L_0 \left( \frac{\alpha+1}{C_1} \right)^{\frac{1}{\alpha+1}}$. \\

4. When $f(j)=C_1 r^j$, then 
\begin{equation}
	F(k)=\frac{ C_1(r^{k+1}-r) }{\ln(r)} , \text{ so } F^{-1}(s) = \frac{1}{\ln(r)}\ln\left( \frac{\ln(r)}{C_1} s+ r\right)-1
\end{equation} 
%Thus, there exists $C, c>0$, such that 
%$$
%E(u,t) < C \exp\left( - \frac{c}{\ln(r)} \ln\left( \frac{\ln(r)}{2L_0+C_1} t + r\right)\right).
%$$
%Simplifying and relabeling $C, c$, this becomes
%$$
%E(u,t) < \frac{C}{\<t\>^c} E(u,0).
%$$
On the other hand, 
\begin{equation}
	B(k) = \frac{C_1 (r^k-1)}{\ln(r)}, \text{ so } B^{-1}(s) = \frac{1}{\ln(r)}\ln\left( \frac{\ln(r)}{C_1} s+ 1\right).
\end{equation}
%Therefore 
%$$
%E(u,t) < C E(u,0) \exp\left(-\frac{c}{\ln(r)} \ln\left( \frac{\ln(r)}{C_1} t + 1\right) \right),
%$$
%cannot hold with $c> 2\Cm L_0$. After simplifying, and relabeling $C$, this becomes
%$$
%E(u,t) < C E(u,0) \frac{1}{\<t\>^{c}},
%$$
%cannot hold with $c> \frac{2\Cm L_0}{\ln(r)}$. 

5. When $f(j)= C_1 e^{j+e^j}$, then 
\begin{equation}
	F(k) = C_1(e^{e^{k+1}}-e^e), \text{ so } F^{-1}(s)= \ln\left(\ln\left( \frac{s}{C_1} + e^e\right)\right)-1.
\end{equation}
%Therefore, there exists $C, c>0$ such that 
%$$
%E(u,t) < C E(u,0) \exp\left(-c\ln\left(\ln\left(\frac{t}{2L_0+C_1}+e^e\right)\right)\right).
%$$
%After simplification and relabeling $C, c$, this gives
%$$
%E(u,t) < C E(u,0) \frac{1}{\ln(2+t)^{c}}.
%$$
On the other hand, 
\begin{equation}
	B(k)=C_1(e^{e^k}-e), \text{ so } B^{-1}(s)= \ln\left(\ln\left(\frac{s}{C_1} + e\right)\right).
\end{equation} %Then 
%$$
%E(u,t) < C E(u,0) \exp\left(-c \ln\left(\ln\left(\frac{s}{C_1}+e\right)\right)\right),
%$$
%cannot hold with $c>2\Cm L_0$. So
%$$
%E(u,t) < C E(u,0) \frac{1}{\ln(s+e)^{c}},
%$$
%cannot hold with $c> 2\Cm L_0$. 

As stated above, in each of these cases, applying parts 1) and 2) of this proposition gives the desired conclusion.
\end{proof}
\subsection{Oscillating Damping, non-trivial for shrinking time intervals}
We now consider another example of damping which alternates in $t$ between $0$ and a non-trivial function. We now consider damping which is non-trivial on shrinking time intervals and $0$ on fixed size time intervals. We allow the damping to depend on $x$, so long as the $x$ dependent piece is bounded from below by a positive constant. 

Consider $\chi \in L^{\infty}([0,1], [0,1])$ with $\chi \equiv 1$ on $[\frac{1}{4}, \frac{3}{4}]$. 
Fix $S_0>0.$ Consider $f: \Nb_0 \ra (0,S_0]$, with $C_m, C_M, \beta>0$ such that $C_m (t+1)^{-\beta} \leq f(t) \leq C_M (t+1)^{-\beta}$. Consider $g \in L^{\infty}(M)$ with $c_w, C_W>0$ such that $c_w < g(x) < C_W$.

For $k \in \Nb_0$, define 
\begin{equation}\label{shrinkingdampdef}
W(x,t) = \begin{cases}
g(x) \chi\left( \frac{t-k S_0}{f(k)} \right) & kS_0 < t < kS_0 +f(k),\\
0 & kS_0 + f(k) \leq t \leq (k+1) S_0.
\end{cases}
\end{equation}
That is, the damping is on for an interval of length at most $f(k)$ starting at $kS_0$ and then is off until $(k+1)S_0$. 
\begin{figure}[h]
\center
\begin{tikzpicture}
\node at (-0.3, 2) {$W$};
\node at (8, -.3) {$t$};
\draw [->] (0,0) -- (0,2.2);
\draw [->] (0,0) -- (8.4,0);
\node at (0,-.3) {$0$};
\node at (2,-.3) {$S_0$};
\node at (4,-.3) {$2S_0$};
\node at (6,-.3) {$3S_0$};
\draw [-] (0,2)--(1,2);
\draw[-] (1,2)--(1,0);
\draw[-] (1,0)--(2,0);
\draw[-] (2,2)--(2,0);
\draw[-] (2,2)--(2.66,2);
\draw[-] (2.66,2)--(2.66,0);
\draw[-] (2.66,0)--(4,0);
\draw[-](4,2)--(4,0);
\draw[-](4,2)--(4.4,2);
\draw[-] (4.4,2)--(4.4,0);
\draw[-] (4.4,0)--(6,0);
\draw[-](6,2)--(6,0);
\draw[-](6,2)--(6.2,2);
\draw[-] (6.2,2)--(6.2,0);
\draw[-](6.2,0)--(8,0);
\node at (.5, 2.5) {$\underbrace{f(0)}$};
\node at (2.35, 2.5) {$\underbrace{f(1)}$};
\node at (4.2, 2.5) {$\underbrace{f(2)}$};
\node at (6.1, 2.5) {$\underbrace{f(3)}$};
\node at (8,1) {$\cdots$};
\end{tikzpicture}
\caption{Time intervals where $W$ satisfying \eqref{shrinkingdampdef} is non-trivial or 0.}
\end{figure}

\begin{proposition}\label{shrinkingdecay}
With damping as in \eqref{shrinkingdampdef},
\begin{enumerate}
	\item If $0 \leq \beta<1/3$, there exists $C,c>0$, such that for all $u$ solving \eqref{TDWE}
	\begin{equation}
		E(u,t) < C E(u,0) \exp(-c t^{1-3\beta}).
	\end{equation}
	\item If $\beta=1/3$, there exists $C,c>0$, such that for all $u$ solving \eqref{TDWE}
	\begin{equation}
		E(u,t) \leq C(t+1)^{-c} E(u,0).
	\end{equation}
	\item If $\beta>1/3$, there exists $c<1$, such that for $t>f(0),$ for all $u$ solving \eqref{TDWE}
	\begin{equation}
		E(u,t) <c E(u,0).
	\end{equation}
\end{enumerate}
If in addition $\chi \in C^0([0,1])$, and $g \in C^0(M)$ then
\begin{enumerate}
	\item If $0\leq \beta<1$, then 
	\begin{equation}
		E(u,t) < C E(u,0) \exp(-c t^{1-\beta}),
	\end{equation} 
	cannot hold for all $u$ solving \eqref{TDWE} if  $c> \frac{2C_M C_W}{(1-\beta)S_0^{1-\beta}}$.
	\item If $\beta=1$, then
	\begin{equation}
		E(u,t) < C E(u,0) \frac{1}{t^{c}} E(u,0),
	\end{equation} 
	cannot hold for all $u$ solving \eqref{TDWE} if $c>2C_M C_W$.
	\item If $\beta>1$, then there is not a uniform stabilization rate.
\end{enumerate}
\end{proposition}
\begin{remark}
	\begin{enumerate}
		\item Note that when $\beta>0$ the damping is 0 for a growing proportion of time, and so the time dependent geometric control condition cannot be satisfied. 
		\item At first the hypothesis on $g(x)$ seems stricter than our spatial hypotheses for the other examples. However,  our assumption on $g(x)$ is equivalent to satisfying Assumption \ref{TGCC} on $[0,T_0]$ for all $T_0>0$. We require this in order to obtain an observability estimate on all of the shrinking intervals where $W$ is non-trivial.
	\end{enumerate}
\end{remark}
This result will follow from a proof analogous to Theorem \ref{bdecaythm}. In particular, a damped wave observability inequality is proved from a wave equation observability inequality on short time intervals $[0,\d]$, Lemma \ref{basicwaveobserve}, and is then converted into a uniform stabilization rate using Lemma \ref{observetodecay}. 

One notable change is that the wave equation observability inequality on short time intervals $[0,\d]$ has a constant that grows like $\d^{-3}$, and as pointed out in the proof of Lemma \ref{basicwaveobserve}, the power on $\d$ cannot be improved. This produces the power on $f(k)$ in the Lemma below, compare to Lemma \ref{decreasingdwobserve}. This is what causes the gap between the upper and lower bounds in the statement of the proposition. 
%The gap between the upper and lower bounds is due to the power on $f(k)$ in the damped wave observability inequality below. That in turn, comes from the short-time observability constant for the wave equation, Lemma \ref{basicwaveobserve}, which cannot be improved in general. 

We begin with a damped wave observability inequality. 
\begin{lemma}\label{shrinkingdwobserve}
Suppose $u$ solves \eqref{TDWE} with damping as in \eqref{shrinkingdampdef}. Then, there exists $C>0$, such that for all $k \in \Nb_0$
\begin{equation}
	C f(k)^3 E(u,kT_0) \leq \int_{kS_0}^{(k+1)S_0} \int_M W|\p_t u|^2 dx dt.
\end{equation}
\end{lemma}
\begin{proof}
Let $\psi_k$ solve \eqref{WEb} with $(\psi_k, \p_t \psi_k)_{t=kS_0}=(u,\p_t u)|_{t=kS_0}$, so $E(u,kS_0) = E(\psi, kS_0)$. 
By the short time wave equation observability inequality, Lemma \ref{basicwaveobserve}, there exists $C>0$, such that for any $\d>0, s>0$
\begin{equation}
 E(\psi_k, kS_0+s) \leq \frac{C}{\d^3}\int_{kS_0+s}^{kS_0+\d+s} \int_M |\p_t \psi_k|^2 dx dt. 
\end{equation}
Now since $\psi_k$ solves the wave equation $E(\psi_k, kS_0)=E(\psi_k, kS_0+s)$. Combining this with $g(x) \geq c_W$ we obtain
\begin{equation}
	E(\psi_k, kS_0) \leq \frac{C}{c_W \d^3} \int_{kS_0 +s}^{kS_0+\d+s} \int_M g(x) |\p_t \psi_k|^2 dx dt.
\end{equation}
Then, letting $\d=\frac{f(k)}{2}, s=\frac{f(k)}{4}$, and using that $W(x,t) \geq g(x)$ on $[kS_0+\frac{f(k)}{4}, kS_0+\frac{3f(k)}{4}]$ we have 
\begin{align}
E(u,kS_0) = E(\psi_k, kS_0)  &\leq \frac{8C}{f(k)^3} \int_{kS_0+\frac{f(k)}{4}}^{kS_0+\frac{3f(k)}{4}} \int_M g(x) |\p_t \psi_k|^2 dx dt
\\ 
&\leq \frac{C}{f(k)^3} \int_{kS_0+\frac{f(k)}{4}}^{kS_0+\frac{3f(k)}{4}} W|\p_t \psi_k|^2 dx dt \\
&\leq \frac{C}{f(k)^3} \int_{kS_0}^{(k+1)S_0} \int_M W |\p_t \psi_k|^2 dx dt.
\end{align}
Now applying Lemma \ref{observeconnection} 
\begin{align}
E(u,kS_0) \leq \frac{C(1+2S_0 C_W)^2}{f(k)^3} \int_{kS_0}^{(k+1)S_0} \int_M W |\p_t u|^2 dx dt.
\end{align}
Multiplying both sides by $\frac{f(k)^3}{C(1+2S_0 C_W)^2}$ gives the desired conclusion. 
\end{proof}

We now combine the damped wave observability inequality of Lemma \ref{shrinkingdwobserve} with Lemma \ref{observetodecay} to obtain uniform stabilization rates. To obtain lower bounds, we estimate $\Sigma(t)$ and apply Theorem \ref{lowerboundthm}.
\begin{proof}[Proof of Proposition \ref{shrinkingdecay}]
Since $C_m (1+t)^{-\beta} \leq f(t)$, then by Lemma \ref{shrinkingdwobserve}, there exists $C>0$ such that 
$$
C \frac{1}{(1+jS_0)^{3\beta}} E(u,jS_0) \leq \int_{jS_0}^{(j+1)S_0} \int_M W |\p_t u|^2 dx dt. 
$$
Then by Lemma \ref{observetodecay}, letting $B(k) = \sum_{j=0}^{k-1} C (1+jS_0)^{-3\beta}$,
$$
E(u,kS_0) \leq E(u,0) \exp(-B(k)).
$$
Now since $(1+zS_0)^{-3\beta}$ is decreasing in $z$, $(1+aS_0)^{-3\beta} \geq \int_a^{a+1} (1+zS_0)^{-3\beta}dz$. Therefore
\begin{align}
B(k) \geq \int_0^k C (1+ z S_0)^{-3\beta} dz = \begin{cases}
\frac{C}{1-3\beta} ((1+kS_0)^{1-3\beta} -1) & \beta \neq 1/3 \\
C \ln(1+kS_0) & \beta = 1/3. 
\end{cases}
\end{align}
Then for $k$ large enough, there exists $c>0$ such that 
\begin{align}
B(k) \geq \begin{cases}
c(1+(k+1)S_0)^{1-3\beta} & \beta < 1/3 \\
c \ln(1+(k+1)S_0) & \beta =1/3 \\
c & \beta >1/3.
\end{cases}
\end{align}
Since $E(u,t)$ is non-increasing in $t$, for $t \in [kS_0, (k+1)S_0]$, there exists $C>0$ such that  
\begin{align}
E(u,t) \leq E(u,kS_0) &\leq E(u,0) \begin{cases} C \exp(-c (1+(k+1)S_0)^{1-3\beta}) & \beta < 1/3 \\ 
C \exp(-c\ln(1+(k+1)S_0)) & \beta =1/3 \\
\exp(-c) & \beta > 1/3 \end{cases} \\
& \leq E(u,0) \begin{cases} C \exp(-c(1+t)^{1-3\beta}) & \beta < 1/3 \\
C (1+t)^{-c} & \beta =1/3 \\
\exp(-c) & \beta > 1/3.
\end{cases}
\end{align}
Since every $t \geq 0$ is contained in such an interval for some $k$, this gives the desired uniform stabilization rates.

2. Fix $t \in (0,\infty)$ and define $k = \lceil \frac{t}{S_0} \rceil$. By the definition of $W$
\begin{align}
\Sigma(t)&=\inf_{\gamma} \int_0^t W(\gamma(s), s) ds \leq C_W \sum_{j=0}^{k} f(j) =C_W \left( \sum_{j=1}^k f(j)+f(0) \right) 
\end{align}
Now since since $f$ is decreasing, $f(a)\leq \int_{a-1}^{a}f(s)ds$, and so 
\begin{align}
\Sigma(t) &\leq C_W \left(\int_0^k f(s) ds +f(0)\right) \\
&\leq C_M C_W \left(\int_0^k \frac{1}{(s+1)^{\beta}} ds+1\right) =
\begin{cases} \frac{C_M C_W}{(1-\beta)}\left( (k+1)^{1-\beta} -\beta \right) &\beta \neq 1 \\
 C_M C_W (\ln(k+1)+1)  &\beta=1. \end{cases} \label{eq:Sigmashortinterval}
\end{align}
Now note that $k \leq \frac{t}{S_0}+1$. Note further, since $t \geq 0$ and for $ 0 \leq \beta <1$
\begin{equation}
	\left(\frac{t}{S_0}+2\right)^{1-\beta} \leq \left(\frac{t}{S_0}\right)^{1-\beta}+2,
\end{equation}
which can be seen by plugging in $t=0$, and separately, comparing the $t$ derivatives of both sides. For $\beta>1$, note that there exists $C>0$, such that for all $k$, 
\begin{equation}
	(1-\beta)^{-1}((k+1)^{1-\beta}-\beta) \leq C.
\end{equation}
Applying these to \eqref{eq:Sigmashortinterval} we obtain
\begin{align}
	\Sigma(t) &\leq \begin{cases}
		\frac{C_M C_W}{1-\beta}\left( \left( \frac{t}{S_0}+2 \right)^{1-\beta}-\beta \right) & \beta<1\\
		C_M C_W \left(\ln \left( \frac{t}{S_0}+2 \right) +1 \right) & \beta=1 \\
		C & \beta >1
	\end{cases} \\
	&\leq \begin{cases}
		\frac{C_M C_W}{1-\beta}\left( \left( \frac{t}{S_0}\right)^{1-\beta}-C \right) & \beta<1\\
		C_M C_W \left(\ln \left( \frac{t}{S_0}+2 \right) +1 \right) &\beta=1 \\
		C & \beta >1. 
	\end{cases}
\end{align}
Applying Theorem \ref{lowerboundthm} gives the desired conclusions.
%When $\beta<1$, $\Sigma(t) \leq \frac{C_M C_W}{(\beta-1)}(k^{1-\beta}-1)$, and since $k=\lfloor \frac{t}{S_0}\rfloor$, by Theorem \ref{lowerboundthm},
%$$
%E(u,t) \leq C E(u,0) \exp(-c t^{1-\beta}),
%$$
%cannot hold for  $c > \frac{2C_M C_W}{(1-\beta)S_0^{1-\beta}}$. \\
%When $\beta=1$, $\Sigma(t) \leq C_M C_W \ln(k+1)$, so by Theorem \ref{lowerboundthm},
%$$
%E(u,t) \leq C E(u,0) \exp(-c \ln(k+1)), 
%$$
%cannot hold for $c > 2C_M C_W$. Since $k=\lfloor \frac{t}{S_0}\rfloor$, 
%$$
%E(u,t) \leq C E(u,0) \frac{1}{t^{c}},
%$$
%cannot hold for $c>2C_M C_W$. \\
%When $\beta>1$, $\Sigma(t) \leq \frac{C_M C_W}{(\beta-1)}<\infty$, so by Theorem \ref{lowerboundthm}, there exists $c \in (0,1),$ such that for all $t_0$, there exists a solution $u$ of \eqref{TDWE} such that 
%$$
%E(u,t_0) \geq c E(u,0).
%$$
\end{proof}

\appendix

\appendix 
\section{Manifolds with Boundary}
\subsection{Generalized Null Bicharacteristics}\label{nullbichar}
This appendix introduces the generalized bicharacteristic flow of \cite{MelroseSjostrand1978}, see also \cite[Chapter 24]{Hormander3}. The exposition is adapted from Section 1 of \cite{LRLTT}.

Let $g^*$ be the dual metric to $g$. On $T^*(M \times \Rb),$ the principal symbol of $\p_t^2 -\Delta_g$ is $p(x,t,\xi, \tau) = -\tau^2 + g_x^*(\xi,\xi)$, where $(\tau, \xi)$ are the fiber variables for $(t,x)$. The Hamilton vector field of $p$ is given by $H_p f = \{p,f\}$. Classical bicharacteristics are the integral curves of $H_p$ in $\Char(p) = \{p=0\}$. For $\p_t^2-\Delta_g$, the projection of classical bicharacteristics onto $M$, using $t$ as a parameter, are exactly unit-speed geodesics on $M$. 

Define $Y = \Omegab \times \Rb$ and 
\begin{equation}
	\Char_Y(p) = \{\rho=(x,t,\xi,\tau) \in T^*(M \times \Rb)\backslash 0: \, x \in \Omegab \text{ and } p(\rho)=0\}.
\end{equation} 
Let $M$ have dimension $d$. Close to the boundary of $\Omega$ we use geodesic normal coordinates $(x', x_d)$ where $x'=(x_1, x_2, \ldots, x_{d-1})$. So $x_d=0$ at $\p\Omega$ and $x_d>0$ in $\mathring{\Omega}$. Set $y=(x,t), y'=(x',t)$ and $y_{d}=x_d$ which provide coordinates near $\p Y= \p \Omega \times \Rb$. Let $\eta=(\eta', \eta_d)$ be the cotangent variables associated to $y=(y',y_d)$. In these coordinates, the principal symbol of the wave operator is
$$
p(y', y_d, \eta', \eta_d) = \eta_d^2 + r(y, \eta'),
$$
where $r$ is a smooth $y_d$-family of tangential differential symbols.  Define $T^* Y= T^* (\Omega \times \Rb)$ and its boundary
$$
\p T^* Y = \{ \rho = (y, \eta) \in T^*(\Omega \times \Rb): y_d =0\}.
$$
Use $r_0$ to denote the restriction of $r$ to $\p T^* Y$, that is, $r_0(y', \eta') = r(y', y_d=0, \eta')$. 

Define 
\begin{equation}
	\Sigma_0 = \{ \rho=(y',y_d,\eta', \eta_d) \in \Char_Y(p): y_d=0\},
\end{equation} 
or $\Sigma_0 = \Char_Y(p) \cap \p T^* Y$. Using local coordinates define the glancing set $\Gc \subset \Sigma_0$ by 
$$
\Gc= \{(y', y_d=0, \eta', \eta_d) \in \Sigma_0: \eta_d=r(y,\eta')=r_0(y',\eta')=0\} .
$$
%and the hyperbolic set as $H=\Sigma_0 \backslash G$. Note that if $\rho=(y', y_n=0, \eta', \eta_n) \in \Sigma_0$ then 
%$$
%\rho \in H \iff \eta_n^2 > 0, \qquad \rho \in G \iff r(y,\eta')=r_0(y', \eta')  =0. 
%$$
%Note, the glancing set is exactly the set of points $\Char_Y(p)$ on the boundary of $T^* Y$ were $H_p$ is tangent to the boundary. 
The glancing set can also be decomposed as $\Gc=\Gc^2 \supset \Gc^3 \supset \cdots \supset \Gc^{\infty}$, where we write $\rho=(y', y_d=0, \eta', \eta_d) \in \Gc^{k}$, c.f. \cite[(24.3.2)]{Hormander3}, if 
$$
p(y', y_d=0, \eta', \eta_d)=0, \quad H_p^j y_d =0, \quad 0 \leq j <k.
$$
Finally, consider $\Gc^2 \backslash \Gc^3$, the glancing set of order precisely 2, and define subsets of it, the diffractive set $\Gc_d^2$ and the gliding set $\Gc_g^2$, c.f. \cite[Definition 24.3.2]{Hormander3}, as 
$$
\rho \in \Gc^2_d \, \,(\text{resp. } \Gc_g^2) \iff \rho \in \Gc^2 \backslash \Gc^3 \,  \text{ and } \, H_p^2 y_d> 0 \, \,(\text{resp.} <0).
$$
Then $\Gc^2 \backslash \Gc^3 = \Gc_d^2 \cup \Gc_g^2$. This decomposition can be continued for higher even orders of glancing points, but is not needed in this paper. 
\begin{definition}\label{genbichar}
A generalized bicharacteristic of $p$ is a differentiable map 
$$
\Rb \backslash \mathcal{B} \ni s \mapsto \gamma(s) \in (\Char_Y(p) \backslash \Sigma_0) \cup \Gc,
$$
satisfying the following properties
\begin{enumerate}
	\item $\gamma'(s) = H_p(\gamma(s))$ if $\gamma(s) \in \Char_Y(p) \backslash \Sigma_0$ or $\gamma(s) \in \Gc^2_d$. 
	\item $\gamma'(s) = H_{r_0}(\gamma(s))$ if $\gamma(s) \in \Gc \backslash \Gc^2_d$. 
	\item  Every $s_0 \in \mathcal{B}$ is isolated, and there exists $\d>0$, such that for $s \in (s_0-\d, s_0) \cup (s_0, s_0+\d)$ then $\gamma(s) \in \Char_Y(p) \backslash \Sigma_0$. Furthermore, the limits $\lim_{s \ra s_0^{\pm}} \gamma(s)=(y^{\pm}, \eta^{\pm})$ exist and $y_d^-=y_d^+ =0, y^-{}'=y^+{}', \eta^-{}'=\eta^+{}',$ and $\eta_d^-=-\eta_d^+$. 
\end{enumerate}
\end{definition}
In case 1 the generalized bicharacteristic is either in the interior, or at a diffractive point. Here it coincides with a segment of a classical bicharacteristic. Case 2 describes how a generalized bicharacteristic enters or leave the boundary $\p T^* Y$ or locally remains in it. Case 3 describes reflections, when a bicharacteristic transversally encounters the boundary.  

For $y$ near the boundary of $Y$, define ${}^b T_y Y$ to be the tangent vector field generated by $\p_{y'}$ and $y_d \p_{y_d}$. Then define the compressed cotangent bundle ${}^bT^*Y=\bigcup_{y \in Y} ({}^b T_yY)^*$, and define the compression map
\begin{align}
&j: T^* Y \ra {}^b T^* Y, \\
&(y, \eta', \eta_d) \mapsto (y, \eta', y_d\eta_d).
\end{align}
Note that 
\begin{itemize}
	\item for $y \in \Omega \times \Rb,$ then ${}^bT^*_yY=j(T^*_y Y)$ is isomorphic to $T^*_y Y=T^*_y(\Omega \times \Rb)$,
	\item for $y \in \p \Omega \times \Rb,$ then ${}^b T^*_y Y= j(T^*_y Y)$ is isomorphic to $T^*_y(\p \Omega \times \Rb)$.
\end{itemize}
The set of points $(y', y_n=0, \eta', 0) \in {}^b T^* Y|_{y_n=0}$ such that $r_0(y', \eta')>0$ is called the elliptic set $E$. Also set $\hat{\Sigma} = j(\Char_Y(p)) \cup E$ and take the cosphere quotient space $S^* \hat{\Sigma} = \hat{\Sigma}/(0,+\infty)$. This is needed in the defect measure construction. 

Define compressed generalized bicharacteristics to be the image under $j$ of the generalized bicharacteristics of Definition \ref{genbichar}. If ${}^b \gamma=j(\gamma)$ is a compressed generalized bicharacteristic, then ${}^b \gamma:\Rb \ra {}^b T^* Y \backslash E$ is a continuous map. Using $t$ as a parameter,  projecting compressed generalized bicharacteristics down to $M$ gives unit-speed generalized geodesics for $\Omega$. Generalized geodesics remain in $\Omegab$.  %In geometric optics the standard terminology for such a projection is ``ray". 

An important feature of compressed generalized bicharacteristics, as shown in \cite{MelroseSjostrand1978}, is the following proposition. 
\begin{proposition}
A compressed generalized bicharacteristic with no point in $\Gc^{\infty}$ is uniquely determined by any one of its points.
\end{proposition}

\subsection{Defect Measure with Boundary}\label{defect}
Following \cite[Appendix A]{LRLTT} we introduce defect measures on manifolds with boundary. See also \cite{Lebeau1996} and \cite{BurqLebeau2001}
\begin{definition}
Define $\Psi_b^m(Y)$ to be made up of operators of the form $R=R_{\inter} + R_{\tan}$ where $R_{\inter}$ is a classical pseudodifferential operator of order $m$, with compact support in $\Omega \times \Rb,$ and $R_{\tan}$ is a classical tangential pseudodifferential operator of order $m$. In the local normal coordinates introduced in Appendix \ref{nullbichar}, $R_{\tan}$ acts only in the $y'$ variables.  
\end{definition}
Let $\sigma(R_{\inter})$ and $\sigma(R_{\tan})$ be the homogeneous principal symbols of $R_{\inter}$ and $R_{\tan}$ respectively. The restrictions $\sigma(R_{\inter})|_{\Char_Y(p)}$ and $\sigma(R_{\tan})|_{\Char_Y(p) \cup T^* (\Rb \times \p \Omega)}$ make sense. Using the compression map $j: T^* Y \ra {}^b T^* Y$ we define
$$
j\left(\sigma(R_{\inter})|_{\Char_Y(p)} + \sigma(R_{\tan})|_{\Char_Y(p) \cup T^* (\p \Omega \times \Rb)} \right) = :\kappa(R),
$$ 
which is a continuous function on $\hat{\Sigma}= j (\Char_Y(p)) \cup E$. Furthermore, by the homogeneity of the symbols $\kappa(R)$ is a continuous function on $S^* \hat{\Sigma} = \hat{\Sigma} / (0, \infty)$. Then by \cite[Section 2.1]{Lebeau1996} and \cite[Proposition 2.5]{BurqLebeau2001}.
\begin{proposition}
Suppose $\{u_n\}$ is a bounded sequence in $H^1(\Omega\times \Rb)$. If $(\p_t^2 -\Delta_g) u_n=0$ and $u_n$ weakly converges to $0$, then there exists a subsequence $\{u_{n_j}\}$ and a positive measure $\mu$ on $S^* \hat{\Sigma}$ such that for any $R \in \Psi^0(Y)$
$$
\lim_{j \ra \infty} \<Ru_{n_j}, u_{n_j}\>_{H^1(\Omega \times \Rb)} \ra \<\mu, \kappa(R)\>.
$$ 
\end{proposition}
This is a generalization of \cite{Gerard1991} and \cite{Tartar1990}.

\section{Wave equation facts}
In this appendix we compile some elementary facts about solutions of the wave equation.

\subsection{Proof of weak continuity for wave equation}
In this appendix we show that solutions of the wave equation are weakly continuous in their initial data. 
\begin{lemma}\label{weakconverge}
Suppose $(\psi_{0,n}, \psi_{1,n}) \in H \times L^2(\Omega)$ has a weak limit $(\psi_0, \psi_1)$ in $H \times L^2(\Omega)$. If $\psi_n$ solves \eqref{WEb} with initial data $(\psi_{0,n}, \psi_{1,n})$ and $\psi$ solves \eqref{WEb} with initial data $(\psi_0,\psi_1)$, then $\psi_n \rhu \psi$ weakly in $L^2(0,T; H)$ and $\p_t \psi_n \rhu \p_t \psi$ weakly in $L^2(0,T;L^2(\Omega))$.
\end{lemma}

\label{waveweakcontinuityapp}
\begin{proof}
Throughout the proof, inner products are taken over $L^2(\Omega)$. To begin
\begin{align}
\p_t \ltwom{\psi_n(t)}^2 = 2 \int_\Omega \psi_n \p_t \psi_n dx &\leq \left( \ltwo{\p_t \psi_n(t)}^2 +  \ltwom{\psi_n(t)}^2 \right)\\
&\leq \left(2 E(\psi_n, t) + \ltwom{\psi_n(t)}^2\right).
\end{align}
Therefore by Gr\"onwall's inequality
\begin{equation}
	\ltwom{\psi_n(t)}^2 \leq e^{t} \left( \ltwom{\psi_n(0)}^2 + \int_0^t 2 E(\psi_n, s) ds \right). 
\end{equation}
Since $E(\psi_n, t)=E(\psi_n,0)$ for all $t$
\begin{align}
\max_{0 \leq t \leq T} \left( \ltwom{\psi_n(t)}^2 + E(\psi_n, t) \right) &\leq (2 Te^{T} +1) \left( E(\psi_n, 0) + \ltwom{\psi_n(0)}^2 \right) \\
&= C \left( \nm{\psi_{0,n}}_{H}^2 + \ltwo{\psi_{1,n}}^2 \right).
\end{align}
Furthermore for any $v \in H$ with $\nm{v}_{H} \leq 1$
\begin{equation}
	\< \p_t^2 \psi_n(t), v(t)\> + \<\nabla \psi_n(t), \nabla v(t)\> =0, 
\end{equation}
so 
\begin{equation}
	|\<\p_t^2 \psi_n(t), v(t)\>| \leq \ltwo{\nabla \psi_n(t)} \leq (2E(\psi_n,t))^{1/2}.
\end{equation}
And thus 
\begin{equation}
	\int_0^T \nm{\p_t^2 \psi_n(t) }_{H'}^2 dt \leq 2T E(\psi_n, 0). 
\end{equation}
Thus there exists $C>0$ such that 
\begin{equation}
\max_{0 \leq t \leq T} \left( \ltwom{\psi_n(t)}^2 + E(\psi_n, t) \right) + \int_0^T \nm{\p_t^2 \psi_n(t)}_{H'}^2 dt \leq C \left( \nm{\psi_{0,n}}_{H}^2 + \ltwo{\psi_{1,n}}^2 \right).
\end{equation}
Therefore $\psi_n$ is bounded in $L^2(0,T; H), \p_t \psi_n$ is bounded in $L^2(0,T; L^2(\Omega))$ and $\p_t^2 \psi_n$ is bounded in $L^2(0,T;H')$. 

Thus there exists $v \in L^2(0,T; H)$ such that, up to replacement by subsequences, $\psi_n \rhu v$ weakly in $L^2(0,T; H),$ $\p_t \psi_n \rhu \p_t v$ weakly in $L^2(0,T;L^2(\Omega))$ and $\p_t^2 \psi_n \rhu \p_t^2 v$ weakly in $L^2(0,T;H')$. The proof will be completed if it is shown that $v$ solves the wave equation and has $(v,\p_t v)|_{t=0} = (\psi_0, \psi_1)$. 

To see this let $g \in C^1([0,T]; H)$. Then since $\psi_n$ solves \eqref{WEb} 
\begin{equation}\label{weakwave}
\int_0^T \< \p_t^2 \psi_n, g\> + \< \nabla \psi_n, \nabla g\> dt =0.
\end{equation}
And so taking the limit as $n \ra \infty$, by the weak convergence of $\p_t^2 \psi_n$ to $\p_t^2 v$ and $\nabla \psi_n$ to $\nabla v$
\begin{equation}
\int_0^T \<\p_t^2 v, g\> + \<\nabla v, \nabla g\> dt =0.
\end{equation}
That is $v$ solves \eqref{WEb}. Note also $v \in C([0,T]; L^2(\Omega)), \p_t v \in C([0,T]; H')$, so it makes sense to evaluate $v, \p_t v$ at $t=0$. Now choose $f \in C^2([0,T]; H)$ with $f(T)=\p_t f(T)=0$. Then replacing $g$ by $f$ in \eqref{weakwave} and integrating by parts twice in $t$
\begin{equation}
 \int_0^T \<\psi_n, \p_t^2 f\> + \<\nabla \psi_n, \nabla f\> dt = \< \p_t \psi_n(0), f(0)\> - \<\psi_n(0), \p_t f(0)\>.
\end{equation}
Similarly 
\begin{equation}
 \int_0^T \<v, \p_t^2 f\> + \<\nabla v, \nabla f\> dt = \< \p_t v(0), f(0)\> - \<v(0), \p_t f(0)\>.
\end{equation}
 By weak convergence of $\psi_n$ to $v$ in $L^2(0,T;H)$ the left hand sides of the two preceding equations are equal after taking the limit as $n \ra \infty$. So
\begin{equation}
 \limn \< \p_t \psi_n(0), f(0)\> - \<\psi_n(0), \p_t f(0)\> = \< \p_t v(0), f(0)\> - \<v(0), \p_t f(0)\>.
\end{equation}
 Now note $\psi_n(0) =\psi_{0,n} \rhu \psi_0$ in $H$ and $\p_t \psi_n(0) = \psi_{1,n} \rhu \psi_1$ in $L^2(\Omega)$ so
\begin{equation}
 \< \psi_1, f(0)\> - \<\psi_0, \p_t f(0)\> = \< \p_t v(0), f(0)\> - \<v(0), \p_t f(0)\>.
\end{equation}
 Since $f(t=0), \p_t f(t=0)$ are arbitrary $(v,\p_t v)|_{t=0}=(\psi_0,\psi_1)$. Therefore $\psi=v$, by uniqueness of weak solutions of the wave equation.
\end{proof}

\subsection{Short Time Observability}
In this appendix we prove an observability inequality when the observability window $Q= \Omega_x \times [0,\d]_t$. In particular, we provide an explicit relationship between the observability constant and $\d$. In general, the observability constant must grow like $1/\d^3$. At an intuitive level, this is because solutions of the wave equation with most of their energy in potential energy, have kinetic energy like $\sin^2(t)$, which for small $t$ is of size $t^2$. Integrating this from $t=0$ to $t=\d$, gives a quantity of size $\d^3$, as shown more precisely in Lemma \ref{trigintlemma}.  

We begin with an estimate of trigonometric functions which is uniform in the frequency $\lambda$.
\begin{lemma}\label{trigintlemma}
Fix $N, \lambda_*>0$. There exists $C>0,$ such that for all $A,B \in \C, \d \in (0,N],$ and $\lambda \geq \lambda_*$
$$
|A|^2 + |B|^2 \leq \frac{C}{\d^3} \int_0^{\d} \left|-A \sin(\lambda t) + B \cos(\lambda t)\right|^2 dt.
$$
Furthermore, the power on $\d$ is sharp when $A=1$, $B=0$. 
\end{lemma}
Showing the estimate holds uniformly in  $\lambda$ requires some care.
\begin{proof}
We consider two cases
\begin{enumerate}
	\item $\lambda \d  \in \left(0, \frac{1}{2}\right)$,
	\item $\lambda \d \geq \frac{1}{2}$.
\end{enumerate}
In case 1, expand, then evaluate the integral to obtain
\begin{align}\label{trigintlemmacase1}
\frac{C}{\d^3} &\int_0^{\d} |-A \sin(\lambda t) + B \cos(\lambda t)|^2 dt \nonumber \\
&\geq \frac{C}{\d^3} \int_0^{\d}   |B|^2 \cos^2(\lambda t) + |A|^2 \sin^2(\lambda t) -2|A||B| \sin(\lambda t) \cos(\lambda t) dt  \nonumber\\
&= \frac{C}{\d^3} \left( \frac{2\lambda \d (|A|^2+|B|^2) +(|B|^2-|A|^2) \sin(2\lambda \d) + 2|A| |B| (\cos(2\lambda \d)-1)}{4\lambda} \right).
\end{align}
By Taylor's theorem with remainder, there exists $\zeta \in (0, 2\lambda \d),$ such that
\begin{equation}
	\sin(2\lambda\d) = 2\lambda \d - \frac{4}{3} (\lambda \d)^3 + \cos(\zeta) \frac{4}{15} (\lambda \d)^5.
\end{equation} 
Since $\lambda \d \in \left(0, \frac{1}{2}\right)$ and $\cos(\zeta) \in [0,1]$
\begin{equation}\label{eq:sinTaylor}
2\lambda \d - \frac{4}{3} \lambda^3 \d^3 \leq \sin(2\lambda \d) \leq 2 \lambda \d - \frac{19}{15} \lambda^3 \d^3.
\end{equation} 
Similary, there exists $\xi \in (0, 2\lambda \d),$ such that 
\begin{equation}
	\cos(2\lambda \d) = 1 - 2 (\lambda \d)^2 + \sin(\xi) \frac{2}{3} (\lambda \d)^4.
\end{equation} 
Since $\sin(\xi) \geq 0$
\begin{equation}\label{eq:cosTaylor}
|\cos(2\lambda \d)-1| \leq 2  \lambda^2 \d^2.
\end{equation} 
Plugging \eqref{eq:sinTaylor} and \eqref{eq:cosTaylor} back into \eqref{trigintlemmacase1}
\begin{align}
\frac{C}{\d^3} &\int_0^{\d} |-A \sin(\lambda t) + B \cos(\lambda t)|^2 dt\\
&\geq \frac{C}{\d^2} \left( \frac{2\lambda \d \left(|A|^2+|B|^2\right)+|B|^2(2\lambda \d - \frac{4}{3} \lambda^3\d^3) - |A|^2 (2\lambda \d - \frac{19}{15} \lambda^3 \d^3) - 2|A||B|(2\lambda^2 \d^2)}{4 \lambda \d}\right) \\
&= \frac{C}{\d^2} \left( |B|^2 (1-\frac{\lambda^2 \d^2}{3}) + \frac{19}{60} |A|^2 \lambda^2 \d^2 -|A||B| \lambda \d\right).
\end{align}
Then, using that $|A||B|\lambda \d \leq \frac{1}{2\e} |B|^2 + \frac{\e}{2}|A|^2 \lambda^2 \d^2,$ with $\e=\frac{18}{30}$, and $\lambda \d <\frac{1}{2}$
\begin{align}
\frac{C}{\d^3} \int_0^{\d} |-A \sin(\lambda t) + B \cos(\lambda t)|^2 dt&\geq \frac{C}{\d^2}\left( |B|^2\left(1-\frac{\lambda^2\d^2}{3}-\frac{30}{36}\right) + |A|^2 \frac{ \lambda^2 \d^2}{60} \right) \\
&\geq \frac{C}{\d^2}\left(\frac{|B|^2}{12} + \frac{|A|^2 \lambda^2 \d^2}{60} \right) \\
&=\frac{C|B|^2}{12\d^2} + \frac{C|A|^2 \lambda^2}{60},
\end{align}
and letting $C>\max\left(12N^2,\frac{60}{\lambda_*^2}\right)$ this is greater than or equal to $|A|^2+|B|^2$, as desired. 

In case 2, we begin with a change of variables $s=\lambda t$
\begin{equation*}
\frac{C}{\d^3} \int_0^{\d} |-A \sin(\lambda t) + B \cos(\lambda t)|^2 dt = \frac{C}{\lambda \d^3} \int_0^{\lambda \d} |-A\sin(s) + B \cos(s)|^2 ds.
\end{equation*}
Now, there exists $\phi \in [0,2\pi)$ such that $-A \sin(t) + B \cos(t) = \Ac \cos(t+\phi)$ where $|\Ac|^2=|A|^2+|B|^2$. Thus it remains to show that there exists $C>0$ such that 
\begin{equation*}
1 \leq \frac{C}{\lambda \d^3} \int_0^{\lambda \d} \cos^2 (t+\phi) dt.
\end{equation*}
When $\lambda \d=\frac{1}{2}$, there exists some $C_2>0$, such that 
\begin{equation}
	2 \int_0^{\frac{1}{2}} \cos^2(t+\phi) dt =\frac{1}{\lambda \d} \int_0^{\lambda \d} \cos^2(t+\phi) dt = C_2.
\end{equation}
%For $\lambda \d \in (1/2, \pi+1/2)$
%\begin{equation*}
%\frac{1}{\lambda \d} \int_0^{\lambda \d} \cos^2(t+\phi) dt \geq \frac{1}{\pi+1/2} \int_0^{1/2} \cos^2(t+\phi) dt = \frac{1}{\pi+1/2} \frac{\Cm}{2} \geq \frac{\Cm}{4\pi}
%\end{equation*}
Now for arbitrary $\lambda \d \geq \frac{1}{2}$, let $k \in \Nb_0,$ be such that $\lambda \d \in [k \pi, (k+1)\pi)$. Then, since $\cos^2(t)$ is $\pi$-periodic and non-negative
\begin{align}
\frac{1}{\lambda \d} \int_0^{\lambda \d} \cos^2(t+\phi) dt &\geq \frac{1}{\lambda \d}  \sum_{j=0}^{k-1} \left(\int_{j\pi}^{j\pi+\frac{1}{2}} \cos^2(t+\phi) dt+ \int_{j\pi+\frac{1}{2}}^{(j+1)\pi} \cos^2(t+\phi) dt\right) \\
& \geq \frac{1}{\lambda \d} k \left( \int_0^{1/2} \cos^2(t+\phi) dt +0 \right) \\
&\geq \frac{k}{\lambda \d} \frac{C_2}{2} \geq \frac{1}{(k+1)\pi} \frac{k C_2 }{2} \geq \frac{C_2}{4\pi}.
\end{align}
Dividing by $\d^2$ gives 
$$
\frac{1}{\lambda \d^3} \int_0^{\lambda \d} \cos^2(t + \phi) dt \geq \frac{C_2}{4 \pi \d^2} \geq \frac{C_2}{4 \pi N^2},
$$
so choosing $C> \frac{4 \pi N^2}{C_2}$ gives the desired conclusion. 

3. To see the power on $\d$ is sharp when $A=1, B=0$ use Taylor's Theorem to expand $\sin^2(\lambda_* t)$ at $t=0$. Then for small $\d>0$, compute
$$
\int_0^\d \sin^2(\lambda_* t) dt = \int_0^{\d} (\lambda_* t)^2 + O(\lambda_*^4  t^4) dt = C \d^3.
$$
\end{proof}
With this preliminary estimate we can now prove a short-time observability inequality for the wave equation.
\begin{lemma}\label{basicwaveobserve}
Suppose $\psi$ solves \eqref{WEb}. For $N>0$, there exists $C>0,$ such that for all $\d \in (0,N]$
$$
E(\psi,0) \leq \frac{C}{\d^3} \int_0^{\d} \int_{\Omega} |\p_t \psi|^2 dx dt,
$$
and the power on $\d$ cannot be reduced. 
\end{lemma}
\begin{proof}
Let $\{v_k\}$ be an orthonormal basis of eigenfunctions of $-\Delta$ on $\Omega,$ with associated eigenvalues $\{\lambda_k^2\}$ with $0 \leq \lambda_k \leq \lambda_{k+1}$. That is, $-\Delta v_k = \lambda_k^2 v_k$ and $B v_k=0$ when $\p \Omega \neq \emptyset$. 

Then via separation of variables, c.f. \cite[Section 5.7]{Rauchbook}, for any $\psi(x,t) \in L^2(\Omega \times [0,\infty))$ solving the wave equation, there exist coefficients $A_k, B_k \in \C$ such that 
$$
\psi(x,t) = \sum_{k=0}^{\infty} v_k(x) \left(A_k \cos(\lambda_k t) + B_k \sin(\lambda_k t) \right).
$$
%\begin{align}
%\nabla \psi(x,t) &= \sum_{k=0}^{\infty} a_k \nabla v_k(x) T_k(t), \\
%\p_t \psi(x,t) &= \sum_{k=0}^{\infty} a_k v_k \lambda_k (-A_k \sin(\lambda_k t) + B_k \cos(\lambda_k t).
%\end{align}
Therefore
\begin{equation}
	\ltwo{\nabla \psi(\cdot, t=0)}^2 =\sum_{k=0}^{\infty} \<-\Delta v_k, v_k\> |A_k\cos(0)+B_k\sin(0)|=
	\sum_{k=0}^{\infty} \lambda_k^2 |A_k|^2 \ltwo{v_k}^2.
\end{equation}
And  
\begin{align}
\ltwo{\p_t \psi(\cdot,t)}^2&= \sum_{k=0}^{\infty}\lambda_k^2  \left|-A_k \sin(\lambda_k t) + B_k \cos(\lambda_k t)\right|^2 \ltwo{v_k}^2. 
\end{align}
So
\begin{equation}\label{shortobsenergy}
E(\psi, 0) = \frac{1}{2} \sum_{k=0}^{\infty} \lambda_k^2   \ltwo{v_k}^2 (|A_k|^2+|B_k|^2).
\end{equation}
And
\begin{equation}\label{shortobs}
 \sum_{k=0}^{\infty} \lambda_k^2 \ltwo{v_k}^2 \frac{C}{\d^3} \int_0^{\d} \left|-A_k \sin(\lambda_k t) +B_k\cos(\lambda_k t)\right|^2 dt = \frac{C}{\d^3} \int_0^\d \int_M |\p_t \psi|^2 dx dt. 
\end{equation}
Then applying Lemma \ref{trigintlemma}, with $\lambda_*=\lambda_0$ to each individual term in the sums in \eqref{shortobsenergy} and \eqref{shortobs}, gives the desired inequality. Note that if $\lambda_0=0$, then we instead choose $\lambda_*=\lambda_1$ to apply the Lemma, and the $\lambda_0=0$ terms drop out of \eqref{shortobsenergy} and \eqref{shortobs}. The sharpness in Lemma \ref{trigintlemma} shows the power on $\d$ cannot be reduced.
\end{proof}

\bibliographystyle{alpha}
{\footnotesize
\bibliography{mybib}}

\begin{thebibliography}{DCRDLA12}

\bibitem[BG97]{BurqGerard1997}
N.~Burq and P.~G\'erard.
\newblock Condition n\'ecessaire et suffisante pour la contr\^olabilit\'e
  exacte des ondes.
\newblock {\em C. R. Math. Acad. Sci. Paris}, 325(7):749--752, 1997.

\bibitem[BL01]{BurqLebeau2001}
N.~Burq and G.~Lebeau.
\newblock Mesures de défaut de compacité, application au système de lamé.
\newblock {\em Annales Scientifiques de l’École Normale Supérieure},
  34(6):817--870, 2001.

\bibitem[BLR92]{BardosLebeauRauch1992}
C.~Bardos, G.~Lebeau, and J.~Rauch.
\newblock Sharp sufficient conditions for the observation, control, and
  stabilization of waves from the boundary.
\newblock {\em SIAM Journal on Control and Optimization}, 30(5):1024--1065,
  1992.

\bibitem[DCRDLA12]{drda12}
J.~De~Cazenove, D.A. Rade, A.M.G. De~Lima, and C.A. Ara{\'u}jo.
\newblock A numerical and experimental investigation on self-heating effects in
  viscoelastic dampers.
\newblock {\em Mechanical Systems and Signal Processing}, 27:433--445, 2012.

\bibitem[G{\'e}r91]{Gerard1991}
P.~G{\'e}rard.
\newblock Microlocal defect measures.
\newblock {\em Communications in Partial Differential Equations},
  16(11):1761--1794, 1991.

\bibitem[Har89]{Haraux1989}
A.~Haraux.
\newblock Une remarque sur la stabilisation de certains systemes du deuxieme
  ordre en temps.
\newblock {\em Portugaliae mathematica}, 46(3):245--258, 1989.

\bibitem[H{\"o}r07]{Hormander3}
L.~H{\"o}rmander.
\newblock {\em The Analysis of Linear Partial Differential Operators III}.
\newblock Springer Berlin, 2007.

\bibitem[HPT19]{HumbertPrivatTrelat2019}
E.~Humbert, Y.~Privat, and E.~Tr{\'e}lat.
\newblock Observability properties of the homogeneous wave equation on a closed
  manifold.
\newblock {\em Communications in Partial Differential Equations},
  44(9):749--772, 2019.

\bibitem[HW08]{HirosawaWirth2008}
F.~Hirosawa and J.~Wirth.
\newblock Cm-theory of damped wave equations with stabilisation.
\newblock {\em Journal of mathematical analysis and applications},
  343(2):1022--1035, 2008.

\bibitem[KK11]{Kenigson2011}
J.~Kenigson and J.~Kenigson.
\newblock Energy decay estimates for the dissipative wave equation with
  space--time dependent potential.
\newblock {\em Mathematical methods in the applied sciences}, 34(1):48--62,
  2011.

\bibitem[KK23]{KeelerKleinhenz2023}
B.~Keeler and P.~Kleinhenz.
\newblock Sharp exponential decay rates for anisotropically damped waves.
\newblock {\em Annales Henri Poincar{\'e}}, 24(5):1561--1595, 2023.

\bibitem[Kle17]{Klein2017}
G.~Klein.
\newblock Best exponential decay rate of energy for the vectorial damped wave
  equation.
\newblock {\em SIAM Journal on Control and Optimization}, 56, 07 2017.

\bibitem[Leb96]{Lebeau1996}
G.~Lebeau.
\newblock Equation des ondes amorties.
\newblock In {\em Algebraic and Geometric Methods in Mathematical Physics:
  Proceedings of the Kaciveli Summer School, Crimea, Ukraine, 1993}, pages
  73--109. Springer Netherlands, Dordrecht, 1996.

\bibitem[LL16]{LaurentLeautaud2016}
C.~Laurent and M.~L{\'e}autaud.
\newblock Uniform observability estimates for linear waves.
\newblock {\em ESAIM: Control, Optimisation and Calculus of Variations},
  22(4):1097--1136, 2016.

\bibitem[LRLTT17]{LRLTT}
J.~Le~Rousseau, G.~Lebeau, P.~Terpolilli, and E.~Tr{\'e}lat.
\newblock Geometric control condition for the wave equation with a
  time-dependent observation domain.
\newblock {\em Analysis \& PDE}, 10(4):983--1015, 2017.

\bibitem[LZ23]{lz23}
H.~Lak and S.M. Zahrai.
\newblock Self-heating of viscous dampers under short-\& long-duration loads:
  Experimental observations and numerical simulations.
\newblock In {\em Structures}, volume~48, pages 275--287. Elsevier, 2023.

\bibitem[Mat77]{Matsumura1977}
A.~Matsumura.
\newblock Energy decay of solutions of dissipative wave equations.
\newblock {\em Proceedings of the Japan Academy, Series A, Mathematical
  Sciences}, 53(7):232--236, 1977.

\bibitem[MN96]{MochizukiNakazawa1996}
K.~Mochizuki and H.~Nakazawa.
\newblock Energy decay and asymptotic behavior of solutions to the wave
  equations with linear dissipation.
\newblock {\em Publications of the Research Institute for Mathematical
  Sciences}, 32(3):401--414, 1996.

\bibitem[MN01]{MochizukiNakazawa2001}
K.~Mochizuki and H.~Nakazawa.
\newblock Energy decay of solutions to the wave equations with linear
  dissipation localized near infinity.
\newblock {\em Publications of the Research Institute for Mathematical
  Sciences}, 37(3):441--458, 2001.

\bibitem[Moc76]{Mochizuki1976}
K.~Mochizuki.
\newblock Scattering theory for wave equations with dissipative terms.
\newblock {\em Publications of the Research Institute for Mathematical
  Sciences}, 12(2):383--390, 1976.

\bibitem[MS78]{MelroseSjostrand1978}
R.B. Melrose and J.~Sj{\"o}strand.
\newblock Singularities of boundary value problems. i.
\newblock {\em Communications on Pure and Applied Mathematics}, 31(5):593--617,
  1978.

\bibitem[MS82]{ms1982}
R.B. Melrose and J.~Sj{\"o}strand.
\newblock Singularities of boundary value problems. ii.
\newblock {\em Communications on Pure and Applied Mathematics}, 35(2):129--168,
  1982.

\bibitem[PS19]{PaunonenSeifert2019}
L.~Paunonen and D.~Seifert.
\newblock Asymptotics for periodic systems.
\newblock {\em Journal of Differential Equations}, 266(11):7152--7172, 2019.

\bibitem[Ral69]{Ralston1969}
J.~Ralston.
\newblock Solutions of the wave equation with localized energy.
\newblock {\em Communications on Pure and Applied Mathematics}, 22(6):807--823,
  1969.

\bibitem[Ral82]{Ralston1982}
J.~Ralston.
\newblock Gaussian beams and the propagation of singularities.
\newblock {\em Studies in Partial Differential Equations, MAA Studies in
  Mathematics}, 23:206--248, 1982.

\bibitem[Rau12]{Rauchbook}
J.~Rauch.
\newblock {\em Partial differential equations}, volume 128.
\newblock Springer Science \& Business Media, 2012.

\bibitem[RT75a]{RauchTaylor1975b}
J.~Rauch and M.~Taylor.
\newblock Decay of solutions to nondissipative hyperbolic systems on compact
  manifolds.
\newblock {\em Communications on Pure and Applied Mathematics}, 28(4):501--523,
  1975.

\bibitem[RT75b]{RauchTaylor1975}
J.~Rauch and M.~Taylor.
\newblock Exponential decay of solutions to hyperbolic equations in bounded
  domains.
\newblock {\em Indiana Univ. Math. J.}, 24(1):79--86, 1975.

\bibitem[Rud76]{BabyRudin}
W.~Rudin.
\newblock {\em Principles of mathematical analysis}.
\newblock McGraw-Hill New York, third edition, 1976.

\bibitem[Sha19]{Shao2019}
A.~Shao.
\newblock On carleman and observability estimates for wave equations on
  time-dependent domains.
\newblock {\em Proceedings of the London Mathematical Society},
  119(4):998--1064, 2019.

\bibitem[Tar90]{Tartar1990}
L.~Tartar.
\newblock H-measures, a new approach for studying homogenisation, oscillations
  and concentration effects in partial differential equations.
\newblock {\em Proceedings of the Royal Society of Edinburgh Section A:
  Mathematics}, 115(3-4):193–230, 1990.

\bibitem[Ues80]{Uesaka1980}
H.~Uesaka.
\newblock The total energy decay of solutions for the wave equation with a
  dissipative term.
\newblock {\em Journal of Mathematics of Kyoto University}, 20(1):57--65, 1980.

\bibitem[VJdL21]{vjdl}
E.C. Vargas~Junior and C.R. da~Luz.
\newblock $\sigma$-evolution models with low regular time-dependent effective
  structural damping.
\newblock {\em Journal of Mathematical Analysis and Applications},
  499(2):125030, 2021.

\bibitem[Wir04]{Wirth2004}
J.~Wirth.
\newblock Solution representations for a wave equation with weak dissipation.
\newblock {\em Mathematical methods in the applied sciences}, 27(1):101--124,
  2004.

\bibitem[Wir06]{Wirth2006}
J.~Wirth.
\newblock Wave equations with time-dependent dissipation i. non-effective
  dissipation.
\newblock {\em Journal of Differential Equations}, 222(2):487--514, 2006.

\bibitem[Wir07]{Wirth2007}
J.~Wirth.
\newblock Wave equations with time-dependent dissipation ii. effective
  dissipation.
\newblock {\em Journal of Differential Equations}, 232(1):74--103, 2007.

\bibitem[Wir08]{wirthperiodic}
J.~Wirth.
\newblock On the influence of time-periodic dissipation on energy and
  dispersive estimates.
\newblock {\em Hiroshima mathematical journal}, 38(3):397--410, 2008.

\end{thebibliography}
\end{document}